\documentclass[11pt,a4paper,twoside,reqno]{amsart}
\usepackage{graphicx,calc,amscd}
\usepackage{epsfig}
\usepackage{float}
\usepackage{labelfig}
\usepackage{amsfonts}
\usepackage{amssymb}

\usepackage{fontenc}
\usepackage{amsmath}
\usepackage{latexsym}

\newcommand{\fnote}[1]{\footnote{\small sharp1}}

\newcommand{\G}{{\mathcal G}}
\newcommand{\N}{{\mathbb N}}

\newcommand{\Z}{{\mathbb Z}}
\newcommand{\R}{{\mathbb R}}

\newcommand{\kk}{\Bbbk}
\newcommand{\mass}{\mathbf{M}}
\newcommand{\M}{\mathbf{M}}

\newcommand{\Tr}{\mathcal{T}}
\newcommand{\can}{{\rm can}}
\newcommand{\eps}{\varepsilon}
\newcommand{\area}{{\rm area}}
\newcommand{\diam}{{\rm diam}}
\newcommand{\dias}{{\rm dias}}
\newcommand{\sys}{{\rm sys}}
\newcommand{\scg}{{\rm scg}}
\newcommand{\vol}{{\rm vol}}
\newcommand{\length}{{\rm length}}
\newcommand{\FR}{{\rm FillRad}}
\newcommand{\FillRad}{{\rm FillRad}}
\newcommand{\w}{{\rm UW}_{1}}
\newcommand{\hs}{{\rm HS}}

\newcommand{\cf}{{\it cf.}}
\newcommand{\ie}{{\it i.e.}}

\numberwithin{equation}{section}

\newtheorem{theorem}{Theorem}[section]
\newtheorem{proposition}[theorem]{Proposition}
\newtheorem{corollary}[theorem]{Corollary}
\newtheorem{lemma}[theorem]{Lemma}

\theoremstyle{definition}
\newtheorem{definition}[theorem]{Definition}
\newtheorem{example}[theorem]{Example}
\newtheorem{remark}[theorem]{Remark}

\long\def\forget#1\forgotten{} %

\title[Diastolic and isoperimetric inequalities]{Diastolic and isoperimetric inequalities on surfaces}

\author[F.~Balacheff]{Florent Balacheff}
\address{Florent Balacheff - Universit\'e des Sciences et Technologies, Laboratoire Paul Painlev\'e, B\^at. M2, 59655 Villeneuve d'Ascq, France}
\email{Florent.Balacheff@math.univ-lille1.fr}

\author[S.~Sabourau]{St\'ephane Sabourau}
\address{
St\'ephane Sabourau - Universit\'e Fran\c{c}ois Rabelais, Tours.
Laboratoire de Math\'ematiques et  Physique Th\'eorique.
CNRS, UMR 6083.
F\'ed\'eration de Recherche Denis Poisson (FR 2964).
Parc de Grandmont, 37400 Tours, France}
\email{sabourau@lmpt.univ-tours.fr}



\thanks{The first author was supported by the Swiss National Science Foundation (grant 20-118014/1) during the redaction of this article. The second author has been partially supported by the Swiss National Science Foundation.}

\begin{document}

\begin{abstract}
We prove an universal inequality between the diastole, defined using a minimax process on the one-cycle space, and the area of closed Riemannian surfaces. Roughly speaking, we show that any closed Riemannian surface can be swept out by a family of multi-loops whose lengths are bounded in terms of the area of the surface. This diastolic inequality, which relies on an upper bound on Cheeger's constant, yields an effective process to find short closed geodesics on the two-sphere, for instance.
We deduce that every Riemannian surface can be decomposed into two domains with the same area such that the length of their boundary is bounded from above in terms of the area of the surface.
We also compare various Riemannian invariants on the two-sphere to underline the special role played by the diastole.\\

\noindent {\scshape R\'esum\'e.} 
Nous d\'emontrons une in\'egalit\'e universelle entre la diastole, d\'efinie par un proc\'ed\'e de minimax sur l'espace des $1$-cycles, et l'aire d'une surface riemannienne ferm\'ee. De mani\`ere informelle, nous prouvons que toute surface riemannienne ferm\'ee peut \^etre balay\'ee par une famille de multi-lacets dont les longueurs sont contr\^ol\'ees par l'aire de la surface. Cette in\'egalit\'e diastolique, qui repose sur une majoration de la constante de Cheeger,  fournit en particulier un proc\'ed\'e effectif pour trouver de courtes g\'eod\'esiques ferm\'ees sur une $2$-sph\`ere. 
Nous d\'eduisons que toute surface riemannienne peut \^etre d\'ecompos\'ee en deux domaines de m\^eme aire dont la longueur du bord commun est major\'ee \`a l'aide de l'aire de la surface. Nous comparons \'egalement divers invariants riemanniens sur la $2$-sph\`ere afin de souligner le r\^ole sp\'ecial jou\'e par la diastole.
\end{abstract}

\maketitle


\forget
Rien n'apparait dans le document !
\forgotten

\section{Introduction}

The topology of a manifold~$M$ and the topology of its loop space~$\Lambda M$ are closely related, through their homotopy groups, for instance.
The critical points of the length or energy functionals on the loop space of a Riemannian manifold can be studied using this connection.
These critical points have a special geometric meaning since they agree with the closed geodesics on the manifold.

From the isomorphism between $\pi_{1}(\Lambda M,\Lambda^0 M)$ and~$\pi_{2}(M)$ for the two-sphere (here, $\Lambda^0 M$ denotes the space of constant curves), G.~D.~Birkhoff proved the existence of a nontrivial closed geodesic on every Riemannian two-sphere using a minimax argument on the loop space (this method was extended in higher dimension by A.~Fet and L.~Lyusternik, \cf~\cite{klin}).

Relationships on the lengths of these closed geodesics have then been investigated.
For instance, C.~Croke~\cite{Cr88} showed that every Riemannian two-sphere~$M$ has a nontrivial closed geodesic of length
\begin{equation} \label{eq:cro}
\scg(M) \leq 31 \sqrt{\area(M)}.
\end{equation}
The notation~`$\scg$' stands for the (length of a) shortest closed geodesic.
The inequality~\eqref{eq:cro} has been improved in~\cite{nr}, \cite{Sa04} and~\cite{rot05}.

The closed geodesic on the two-sphere obtained in C.~Croke's theorem does not always arise from a minimax argument on the loop space, even though such an argument is used in the proof.
Indeed, the closed geodesics obtained from a minimax argument on the loop space have a positive index when they are nondegenerate, which is the case for generic (bumpy) metrics.
Now, consider two-spheres with constant area and three spikes arbitrarily long (\cf~\cite[Remark~4.10]{Sa04} for further detail).
On these spheres, the closed geodesics satisfying the inequality~\eqref{eq:cro} have a null index while the closed geodesics with positive index are as long as the spikes.
Therefore, a minimax argument on the loop space does not provide an effective way to bound the area of a Riemannian two-sphere from below.
More precisely, given a Riemannian two-sphere~$M$, we define the diastole, $\dias_{\Lambda}(M)$, over the loop space~$\Lambda M$ as
$$
\dias_{\Lambda}(M) := \inf_{(\gamma_{t})} \sup_{0 \leq t \leq 1} \length(\gamma_t)
$$
where $(\gamma_{t})$ runs over the families of loops inducing a generator of $\pi_{1}(\Lambda M,\Lambda^0 M) \simeq \pi_{2}(M)$.
Then, 
\begin{equation} \label{eq:fail}
\mbox{the ratio }\frac{\dias_{\Lambda}(M)}{\sqrt{\area(M)}} \mbox{ is unbounded}
\end{equation}
on the space of Riemannian metrics on~$M$.

\forget
there exists \emph{no} universal constant~$C$ such that
\begin{equation} 
\dias_{\Lambda}(M) \leq C \, \sqrt{\area(M)}
\end{equation}
for every Riemannian two-sphere~$M$.
That is, there exists a sequence~$M_n$ of Riemannian two-spheres such that
\begin{equation} \label{eq:fail}
\lim_{n \to +\infty} \frac{\dias_{\Lambda}(M_n)}{\sqrt{\area(M_n)}} = + \infty.
\end{equation}
\forgotten

Note that a diastolic inequality between the diastole over the loop space and the volume of the convex hypersurfaces in the Euclidean spaces does hold true, \cf~\cite{tr85}, \cite{Cr88}.

The existence of a closed geodesic on the two-sphere can also be proved by using a minimax argument on a different space, namely the one-cycle space~$\mathcal{Z}_{1}(M;\Z)$.
(Recall that one-cycles are unions of loops, \cf~Section~\ref{sec:dias} for a precise definition.)
This minimax argument relies on F.~Almgren's isomorphism~\cite{almgren} between $\pi_{1}(\mathcal{Z}_{1}(M;\Z),\{0\})$ and $H_{2}(M;\Z)$, which holds true for every closed manifold~$M$.
Given a closed Riemannian surface~$M$, we define the diastole over the one-cycle space as
\begin{equation} \label{eq:Z}
\dias_{\mathcal{Z}}(M) := \inf_{(z_{t})} \sup_{0 \leq t \leq 1} \M(z_t)
\end{equation}
where $(z_{t})$ runs over the families of one-cycles inducing a generator of
\linebreak
$\pi_{1}(\mathcal{Z}_{1}(M;\kk),\{0\})$ and $\M(z_t)$ represents the mass (or length) of~$z_t$. Here $\kk=\Z$ if $M$ is orientable and $\kk=\Z/2\Z$ otherwise.
For short, we will write $\dias(M)$ for $\dias_{\mathcal{Z}}(M)$.
From a result of J.~Pitts \cite[p.~468]{pit74}, \cite[Theorem~4.10]{pitts} (see also~\cite{cc}), this minimax principle gives rise to a union of closed geodesics (counted with multiplicity) of total length~$\dias(M)$.
Hence, $\scg(M) \leq \dias(M)$.
This principle has been used in~\cite{cc}, \cite{nr}, \cite{Sa04}, \cite{rot05}, \cite{rot06}, \cite{Ba08} and \cite{Sa09} in the study of closed geodesics on Riemannian two-spheres.

On nonsimply connected surfaces, no minimax principle is required to show the existence of a closed geodesic.
In this setting, we directly define the systole as the minimum of the length functional over the connected components not containing the trivial loops of the loop space.
Every closed Riemannian surfaces~$M$ of genus~$g \geq 1$ satisfies the following asymptotically optimal systolic inequality
\begin{equation} \label{eq:sys}
\sys(M) \leq C \, \frac{\log (g+1)}{\sqrt{g}} \, \sqrt{\area(M)}
\end{equation}
where $\sys(M)$ is the systole of~$M$ and~$C$ is a universal constant, \cf~\cite{gro83}, \cite{bal04} and~\cite{KS05} for three different proofs. \\

The goal of this article is to establish curvature-free inequalities similar to~\eqref{eq:cro} and~\eqref{eq:sys}.
The use of the one-cycle space, rather than the loop space, introduces some flexibility.
It allows us to cut and paste closed curves, and to deal with both simply and nonsimply connected surfaces.
We show a difference of nature between the diastole over the loop space and the diastole over the one-cycle space on the two-sphere, and between the systole and the diastole for surfaces of large genus.
More precisely, we obtain the following diastolic inequality.

\begin{theorem} \label{theo:A}
There exists a positive constant $C \leq 10^8$ 
such that every closed Riemannian surface~$M$ of genus~$g \geq 0$ satisfies
\begin{equation} \label{eq:A}
\dias(M) \leq C \, \sqrt{g+1} \sqrt{\area(M)}.
\end{equation}
\end{theorem}

Since the minimax principle~\eqref{eq:Z} gives rise to a union of closed geodesics of length~$\dias(M)$, Theorem~\ref{theo:A} yields a construction of {\em short} closed geodesics on surfaces through Morse theory over the one-cycle space.
In particular, it sheds some light on the nature of the closed geodesics whose lengths provide a lower bound on the area of the two-spheres. 
Compare with~\eqref{eq:cro} and~\eqref{eq:fail}.

The dependence on the genus in the inequality~\eqref{eq:A} is optimal, \cf~Remark~\ref{rem:g}, and should be compared with the one in~\eqref{eq:sys}.


A version of the diastolic inequality~\eqref{eq:A} holds true for compact surfaces with boundary, in particular for disks, \cf~Remark~\ref{rem:boundary}.

The inequality \eqref{eq:A} is derived from a stronger estimate, where the diastole is replaced with the technical diastole introduced in Definition~\ref{def:dias}. This estimate also yields the following result.

\begin{corollary} \label{cor:A}
There exists a positive constant $C$ such that every closed Riemannian surface $M$ of genus~$g \geq 0$ decomposes into two domains with the same area whose length of their common boundary~$\gamma$ satisfies
\begin{equation}
\length(\gamma) \leq  C \, \sqrt{g+1} \sqrt{\area(M)}.
\end{equation}
\end{corollary}

The two domains in the previous result are not necessarily connected, even on two-spheres.
A counterexample is given by two-spheres with three long fingers, whose area is equally concentrated in the tips of these fingers (\cf~\cite[Remark~4.10]{Sa04}).

The second author~\cite{Sa04} showed that the area of a bumpy Riemannian two-sphere can be bounded from below in terms of the length of its shortest one-cycle of index one.
Our inequality~\eqref{eq:A} on the two-sphere extends this result since the minimax process used in the definition of the diastole gives rise to a one-cycle of index one when the metric is bumpy.
However, the relation between the filling radius of a bumpy Riemannian two-sphere and the length of its shortest one-cycle of index one established in~\cite{Sa04} cannot be extended to the diastole.
Indeed, from~\cite[Theorem~1.6]{Sa04}, there exists a sequence~$M_n$ of Riemannian two-spheres such that
$$
\lim_{n \to +\infty} \frac{\FR(M_n)}{\dias(M_n)}=0.
$$
This result illustrates the difference of nature between the length of the shortest closed geodesic or of the shortest one-cycle of index one, which can be bounded by the filling radius on the two-sphere, and the diastole.
It also shows that the proof of Theorem~\ref{theo:A} requires different techniques. 

The arguments used throughout this article are rather ``elementary" and robust.
It is our hope that they can be adapted to investigate further problems.

At the end of this article, we consider a minimax principle on the one-cycle space for another functional and establish further geometric inequalities on the two-sphere.

Higher dimensional analogs of the diastole, sometimes called $k$-widths, have been investigated in \cite{alm65}, \cite{pitts}, \cite{gro83} and \cite{Gu07}. In particular, L. Guth \cite{Gu07} obtained upper bounds on the $k$-width of Euclidean domains in terms of their $n$-dimensional volumes. He also showed that no such bound holds in the Riemannian setting for the $(n-1)$-width with $n\geq 3$. Our main result provides a positive result on Riemannian surfaces, answering a question raised in \cite[p. 1148]{Gu07} in a particular case.\\

Let us present the structure of the article and an outline of the proof of the diastolic inequality~\eqref{eq:A}.
The one-cycle space and the definition of the diastole are presented in Section~\ref{sec:dias}.
In Section~\ref{sec:equil}, we show how to replace a Riemannian surface with a simplicial piecewise flat surface with comparable area and diastole.
This will enable us to prove the main diastolic inequality by induction on the number of simplices in an approaching piecewise flat surface.
In Section~\ref{sec:ch}, we prove an upper bound on Cheeger's constant, and on some related invariant, in terms of the area of the surface alone.
This upper bound yields an isoperimetric inequality which permits us to split the surface into two domains $D_1$ and~$D_2$ whose boundary lengths are small in comparison to their areas.
This isoperimetric inequality can be thought of as a discrete version of the diastolic inequality.
At this stage, we could cap off the two domains $D_1$ and~$D_2$ and apply the discrete version of the diastolic inequality to the two resulting surfaces with the hope to derive a continuous version of it by iterating this principle sufficiently many times.
Unfortunately, passing from a discrete parameter family of one-cycles to a continuous parameter family on a smooth surface is technically challenging.
Arguing by induction on the number of simplices of a piecewise flat surface approaching the initial surface turns out to be more manageable in this case.
More specifically, we argue as follows.
In Section~\ref{sec:ch2}, we prove a simplicial version of the isoperimetric inequality obtained from the upper bound on Cheeger's constant.
The two simplicial surfaces obtained by coning off the domains $D_1$ and~$D_2$ have fewer simplices than the initial simplicial surface.
Using a cut-and-paste argument on the families of one-cycles, we compare the diastole of the two new simplicial surfaces to the diastole of the initial simplicial surface in Section~\ref{sec:max}.
The diastolic inequality is then established by induction on the number of simplices and the genus in Section~\ref{sec:proof}.
Further related geometric inequalities about families of one-cycles sweeping out a Riemannian two-sphere are presented in Section~\ref{sec:comp} and in the appendix.

\section{Almgren's isomorphism and definition of the diastole} \label{sec:dias}

We introduce the space of flat chains and cycles with coefficients in $\Z$ or~$\Z/2\Z$, defined in~\cite{fle66} and~\cite{fed}.

\medskip

Let $M$ be a closed Riemannian manifold, $N$ be a submanifold of~$M$ (possibly with boundary), $\kk=\Z$ if $M$ is orientable and $\kk=\Z/2\Z$ otherwise, and $k$ be an integer.

A Lipschitz $k$-chain of~$M$ with coefficients in~$\kk$ is a finite sum $C=\sum a_{i} f_{i}$ where $a_{i} \in \kk$ and~$f_{i}$ is a Lipschitz map from the standard $k$-simplex to~$M$.
The volume of~$C$, denoted by~$|C|$, is defined as $\sum |a_i| \vol(f_i^* \G)$ where $\G$ is the Riemannian metric on~$M$.
The flat pseudo-norm of a Lipschitz $k$-chain~$C$ is the infimum of $|C-\partial D|+|D|$ over the Lipschitz $(k+1)$-chains~$D$.
Two Lipschitz chains $C_{1}$ and~$C_{2}$ are equivalent if the flat pseudo-norm between them is zero.
The flat pseudo-norm induces a distance (flat norm) on the space of equivalent classes of Lipschitz $k$-chains.
By definition, the completion of this metric space, denoted by~$C_{k}(M;\kk)$, is the space of (flat) $k$-chains with coefficients in~$\kk$.

The mass functional on~$C_{k}(M;\kk)$, denoted by~$\M$, is defined as the largest lower-semicontinuous extension of the volume functional on the space of Lipschitz $k$-chains.
The flat norm of a $k$-chain can be computed by using the definition of the flat pseudo-norm with $D$ running over the $(k+1)$-chains and the mass~$\M$ replacing the volume functional.

The space of $k$-chains of~$M$ relative to~$N$ with coefficients in~$\kk$ is defined as the quotient $C_{k}(M,N;\kk):=C_{k}(M;\kk)/C_{k}(N;\kk)$ endowed with the quotient flat norm and the quotient mass still denoted~$\M$.
The boundary map of a Lipschitz $k$-chain induces a boundary map $\partial:C_{k+1}(M,N;\kk) \to C_{k}(M,N;\kk)$ for every $k \in \N$ such that $\partial^2=0$ which gives rise to a complex
$$
\ldots \rightarrow C_{k+1}(M,N;\kk) \overset{\partial}{\rightarrow} C_{k}(M,N;\kk) \rightarrow \ldots
$$
The cycles of this complex are by definition the $k$-cycles of~$M$ relative to~$N$ with coefficients in~$\kk$.
They can be represented by $k$-chains of~$M$ with boundary lying in~$N$.
We denote by $\mathcal{Z}_{k}(M,N;\kk)$ the space formed of these $k$-cycles.
When $N$ is empty, we simply write~$\mathcal{Z}_k(M;\kk)$ or~$\mathcal{Z}_k(M)$. \\

We describe now the structure of the zero- and one-cycles.
A zero-cycle is indecomposable if it induced by a point~$p$, in which case, its mass is equal to~$1$.
Every zero-cycle~$z$ decomposes into a finite sum of indecomposable zero-cycles $p_i$, \ie, $z = \sum_{i=1}^n p_i$, such that $\M(z) = n$.
The structure of the one-cycles is the following.
A one-cycle is indecomposable if it is induced by a simple closed curve or a simple arc with endpoints in~$N$.
Every one-cycle~$z$ decomposes (not necessarily in a unique way) into a sum of indecomposable one-cycles $z_i$, \ie, $z = \sum_{i \in \N} z_i$, such that $\M(z) = \sum_{i \in \N} \length(z_i)$ (see \cite[p. 420]{fed}). \\

The homotopy groups of the space of (relative) cycles have been determined by F. Almgren, \cf~\cite{almgren}, \cite[\S13.4]{alm65} and \cite[\S4.6]{pitts}.
On surfaces, one has the following natural isomorphism for the one-cycle space
\begin{equation} \label{eq:alm}
\pi_1(\mathcal{Z}_1(M;\kk),\{0\}) \simeq H_2(M;\kk) \simeq \kk.
\end{equation}
More generally, if $N$ is a submanifold with boundary of~$M$, one has
\begin{equation} \label{eq:isom}
\pi_1(\mathcal{Z}_1(M,N;\kk),\{0\}) \simeq H_2(M,N;\kk).
\end{equation}

\forget
Strictly speaking, this isomorphism has been established for the space of integral cycles, but the proof still holds true for the space of cycles with $\Z/2\Z$ coefficients.
\forgotten

This isomorphism permits us to apply the Almgren-Pitts minimax principle to the one-cycle space of surfaces.
Recall that $\kk=\Z$ if the surface is orientable and that $\kk=\Z/2\Z$ otherwise. \\

Let us consider the one-parameter families $(z_t)_{0 \leq t \leq 1}$ of one-cycles sweeping out a closed surface~$M$, that is, which satisfy the following conditions
\begin{enumerate}
\item[(D.1)] $z_t$ starts and ends at the null one-cycle;
\item[(D.2)] $z_t$ induces a generator of
$\pi_1(\mathcal{Z}_1(M;\kk),\{0\}) \simeq \kk$.
\end{enumerate}
For technical reasons, we will also consider the following conditions
\begin{enumerate}
\item[(D.3)] there exists a finite subdivision $t_{1} < t_{2} < \cdots < t_{k}$ of~$[0,1]$ and homotopies $(\gamma_{j,t})_{t_p \leq t \leq t_{p+1}}$, $j \in J_{p}$, of finitely many piecewise smooth loops such that the one-cycle~$z_{t}$ agrees with the finite sum $\sum_{j \in J_p} \gamma_{j,t}$ on~$[t_{p},t_{p+1}]$;
\item[(D.4)] for every $t \in ]t_{p},t_{p+1}[$, the loops~$\gamma_{j,t}$, $j \in J_{p}$, are simple and disjoint;
\item[(D.5)] the loops $\gamma_{j,t}$ and~$\gamma_{j',t'}$ are disjoint for $t \neq t'$.
\end{enumerate}

\begin{example}
Let $f:M \to \R$ be a Morse function on a closed surface~$M$.
The level curves $z_t=f^{-1}(t)$
define a one-parameter family of one-cycles~$(z_t)$ satisfying the conditions~(D.1-5).
\end{example}

\begin{definition} \label{def:dias}
The {\em diastole} of $M$, denoted by $\dias(M)$, is defined as the minimax value
\begin{equation}
\dias(M) := \inf_{(z_{t})} \sup_{0 \leq t \leq 1} \M(z_t)
\end{equation}
where  $(z_{t})$ runs over the families of one-cycles satisfying the conditions~\mbox{(D.1-2)} above.
Similarly, the {\em (technical) diastole} of $M$, denoted by $\dias'(M)$, is defined as the minimax value
\begin{equation}
\dias'(M) := \inf_{(z_{t})} \sup_{0 \leq t \leq 1} \M(z_t)
\end{equation}
where  $(z_{t})$ runs over the families of one-cycles satisfying the conditions~\mbox{(D.1-5)} above.
\end{definition}

\begin{remark}
We will prove that the diastolic inequality \eqref{eq:A} holds true if one replaces the diastole~$\dias(M)$ with the (technical) diastole~$\dias'(M)$.
Since $\dias(M) \leq \dias'(M)$, this will yield Theorem~\ref{theo:A}.
\end{remark}

\begin{remark} \label{rem:boundary}
If $M$ is a Riemannian compact surface with boundary~$\partial M$, the isomorphism~\eqref{eq:isom} allows us to define a relative notion of diastole using one-cycles relative to~$\partial M$.
This relative diastole on~$M$ is less or equal to the diastole on the closed Riemannian surface obtained by filling the boundary components of~$M$ with disks of small area.
Therefore, a diastolic inequality similar to~\eqref{eq:A} holds for compact surfaces with boundary where the diastole is replaced with the diastole relative to the boundary.
\end{remark}

\forget
\section {Comparison between the systole and the diastole
of non simply connected surfaces}

We will prove that
\begin{proposition}
If $M$ is a Riemannian closed surface with $\pi_1(M)\neq 0$, then
$$
\sys(M)\leq \dias(M)
$$
where $\sys(M):=\inf\{\length(\gamma) \mid \gamma \text{ is non
contractible}\}$.
\end{proposition}

\begin{proof}
Fix $\eps>0$ and  let $\{z_t\}_{t \in [0,1]}$ be a family of
one-cycles such that
$$
\sup_{t \in [0,1]} \M(z_t) \leq \dias(M)+\eps.
$$
If there exists a $t \in [0,1]$ such that $z_t$ decomposes into a
sum of indecomposable cycles $\{z_{i,t}\}_{i \in \N}$ such that one
of these cycles is non contractible, say $z_{i_0,t}$ for some $i_0
\in \N$, then $\sys(M) \leq \M(z_{i_0,t}) \leq \M(z_t) \leq
\dias(M)+\eps$.

So we can suppose that each decomposition of every $z_t$ gives rise
to a sum of contractible closed curves. Fix $t \in [0,1]$ and such a
decomposition $\{z_{i,t}\}_{i \in \N}$. Then each $z_{i,t}$ bounds a
disk $D^2_{i,t}$. This define a $2$-cycle $D^2_t=\sum_{i \in \N}
D^2_{i,t}$. The problem is that this family is not necessarily
continuous in $t$. If this is the case, it is then easy to end the
proof: denote by $C$ the  $3$-cycle induced by the family
$\{D^2_t\}_{t\in [0,1]}$ from the isomorphism
$$
\pi_1(\mathcal{Z}_2(M;\kk)) \simeq \mathcal{Z}_3(M;\kk).
$$
Then $\partial C=z$ which proves that $[z]=0$. A contradiction.

\end{proof}
\forgotten

\section{From Riemannian metrics to simplicial surfaces} \label{sec:equil}

\begin{proposition} \label{prop:bilip}
Every closed Riemannian surface~$M$ is $K$-bilipschitz homeomorphic to a closed piecewise flat surface~$M_{0}$ triangulated by equilateral flat triangles, where $K=33$.
\end{proposition}

\begin{remark}
If we do not require the triangles of~$M_{0}$ to have the same size, we can take the constant $K$ to be arbitrarily close to one in the previous proposition.
However, this is no longer true for surfaces with nonzero Euler caracteristic when all the triangles of~$M_{0}$ are supposed to be equilateral. 

Let us argue by contradiction.
Subdividing the triangulation of~$M_{0}$ if necessary, we can assume that the equilateral flat triangles of~$M_{0}$ are arbitrarily small.
Consider the geodesic triangulation~$\Tr$ of~$M$ isotopic to the image of the triangulation of~$M_{0}$ under the $K$-bilipschitz homeomorphism while keeping the vertices fixed.
If $K$ is close enough to one, the angles of all the triangles of~$\Tr$ are close to~$\frac{\pi}{3}$.

From Euler's formula, we have the following relation
$$
6 \chi(M) = \sum_{i} (6 - v_{i})
$$
between the Euler characteristic of~$M$ and the degrees~$v_{i}$ of the vertices of its triangulation.
If the Euler characteristic of~$M$ is nonzero, at least one vertex of the triangulation of~$M$ has its degree different from~$6$.
In this case, the angles of the triangulation~$\Tr$ of the smooth surface~$M$ cannot all be close to~$\frac{\pi}{3}$.
Hence a contradiction. 

For the two-torus and the Klein bottle, the constant $K$ in the previous proposition can be taken to be arbitrarily close to one, see \cite{cvm} and the following. 
\end{remark}

The proof of Proposition~\ref{prop:bilip} rests on the following construction.

\medskip

Let $M$ be a closed Riemannian surface.
Choose~$\varepsilon \in ]0,\frac{1}{100}[$ and $\eps' \in ]0,\varepsilon[$.
From \cite{cvm}, there exists a geodesic triangulation~$\Tr$ of $M$ such that
\begin{enumerate}
\item the diameter of every triangle of~$\Tr$ is less than~$\eps'$;
\item the angles of every triangle of~$\Tr$ are between~$\frac{2\pi}{7} -\eps'$ and~$\frac{5\pi}{14} + \eps'$.
\end{enumerate}
Denote by~$\Delta_i$ the triangles of the triangulation~$\Tr$.
By taking~$\eps'$ small enough, we can assume that every triangle~$\Delta_i$ of~$\Tr$ is almost isometric  to a triangle~$\Delta_i'$ of the Euclidean plane with the same side lengths as~$\Delta_i$.
More precisely, we can assume that

\begin{enumerate}
\item[(a)] there exists a $(1+\eps)$-bilipschitz homeomorphism between $\Delta_i$ and~$\Delta_i'$;
\item[(b)] the diameter of the triangle~$\Delta_i'$ is less than $\eps$;
\item[(c)] the angles of the triangle~$\Delta_i'$ are strictly between $\frac{\pi}{4}$ and~$\frac{3\pi}{7}$.
\end{enumerate}

Replacing the triangles~$\Delta_i$ of the triangulation~$\Tr$ of $M$ by the Euclidean triangles~$\Delta_i'$ gives rise to a piecewise flat metric with conical singularities on the surface.
The size of the Euclidean triangles making up this piecewise flat metric may vary from one triangle to another.
By subdividing some of these triangles, one can make this size more uniform.

More precisely, the segments connecting the midpoints of the sides of each Euclidean triangle~$\Delta_i'$ decompose~$\Delta_i'$ into four triangles twice smaller than~$\Delta_i'$ with the same angles as~$\Delta_i'$ (see Figure \ref{divtr}).

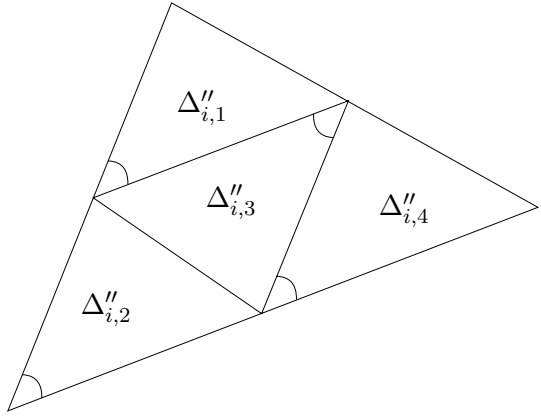
\begin{figure}[h]
\leavevmode \SetLabels
\L(.4*.73) $\Delta''_{i,1}$\\
\L(.3*.24) $\Delta''_{i,2}$\\
\L(.43*.5) $\Delta''_{i,3}$\\
\L(.61*.48) $\Delta''_{i,4}$\\
\endSetLabels
\begin{center}
\AffixLabels{\centerline{\epsfig{file =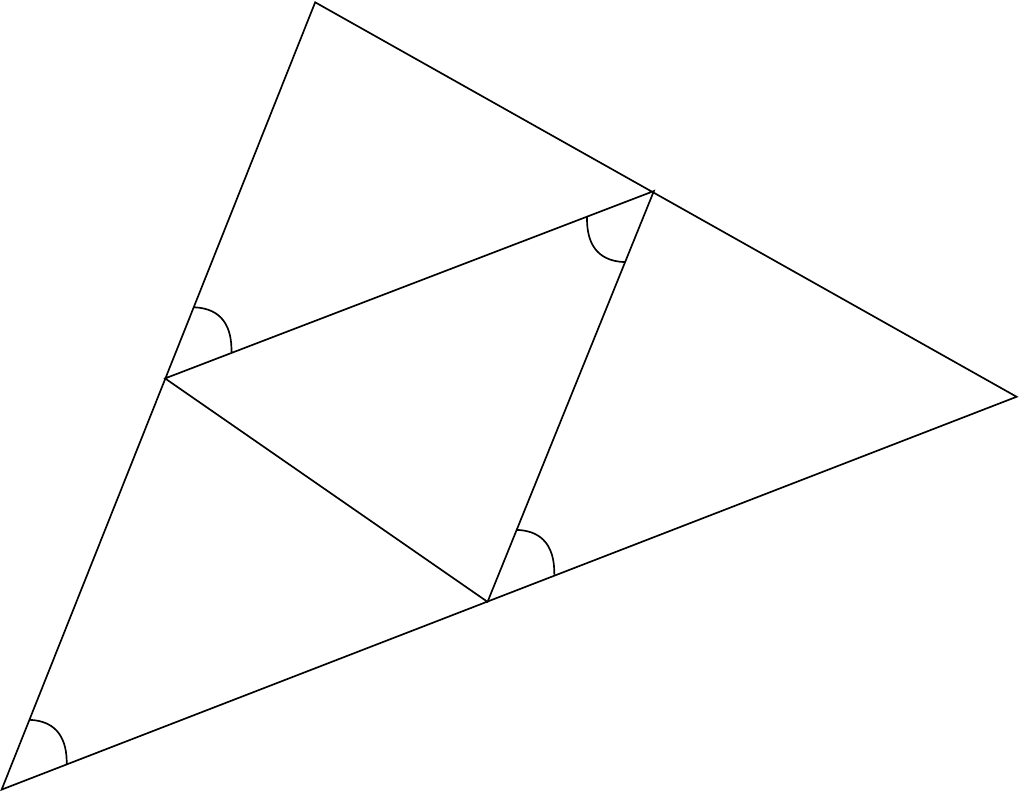,width=7cm,angle=0}}}
\end{center}
\caption{Decomposition of a triangle $\Delta'_i$} \label{divtr}
\end{figure}

By iterating this subdivision process, replacing some of the~$\Delta_i'$'s by four triangles twice smaller than~$\Delta_i'$, one can construct a new collection~$\Tr''$ of triangles~$\Delta_j''$ such that
$$
\max_j \diam (\Delta_j'')  \leq 2 \, \min_j \diam (\Delta_j'').
$$


\begin{lemma} \label{lem:bilip1}
There exists $\ell > 0$ such that every triangle of~$\Tr''$ is $8$-bilipschitz homeomorphic to the equilateral flat triangle of side length~$\ell$ through an affine map.
\end{lemma}

\begin{proof}
Let $\ell = \max_j \diam (\Delta_j'')$ and $\Delta_j''$ be a triangle of~$\Tr''$. By construction, the angles of~$\Delta_j''$ strictly lie between~$\frac{\pi}{4}$ and~$\frac{3\pi}{7}$ and the length of its longest side between~$\frac{\ell}{2}$ and~$\ell$.

Consider the longest side of~$\Delta_j''$. Let $\delta$ be the perpendicular bisector of this side and $v$ be the vertex of~$\Delta_j''$ opposite to this side.
Move $v$ to its orthogonal projection on~$\delta$.
Then, move this projection along the perpendicular bisector~$\delta$ so as to obtain an equilateral flat triangle.
A straightforward computation shows that the composite of these two transformations defines a $4$-bilipschitz affine homeomorphism from~$\Delta_j''$ onto an equilateral flat triangle (recall that the angles of~$\Delta_j''$ are greater than~$\frac{\pi}{4}$).

The length of the sides of this equilateral triangle is between $\frac{\ell}{2}$ and~$\ell$.
Therefore, a homothety of ratio between $1$ and~$2$ transforms this equilateral triangle into an equilateral flat triangle of side length~$\ell$.
\end{proof}

The collection~$\Tr''$ of these triangles may not form a triangulation of~$M$ since the vertex of a triangle~$\Delta_j''$ may lie in the interior of the edge of another triangle.
However, the union of the triangles of~$\Tr''$  covers~$M$ and the interiors of two distinct triangles of~$\Tr''$ are disjoint.
Furthermore, if~$e_1$ and~$e_2$ are the edges of two triangles of~$\Tr''$ then either~$e_1$ lies in~$e_2$ or~$e_2$ lies in~$e_1$, unless the two edges are disjoint or intersect at a single point.

Let us see now how to obtain a genuine triangulation of~$M$.
For that purpose, we will need the following lemma.

\forget
\begin{lemma} \label{lem:1or2}
Consider an edge of a triangle of~$\Tr''$ maximal for inclusion.
This edge agrees with either the edge of another triangle of~$\Tr''$ or the union of the edges of exactly two other triangles of~$\Tr''$.
\end{lemma}
\forgotten

\begin{lemma} \label{lem:1or2}
Consider an edge of a triangle of~$\Tr''$ which contains at least two edges of some other triangles of~$\Tr''$.
Then this edge agrees exactly with the union of these two edges.
\end{lemma}

\begin{proof}
Since the angles of every triangle~$\Delta''$ of~$\Tr''$ are greater than~$\frac{\pi}{4}$,
the length of the longest side of~$\Delta''$, which agrees with its diameter, is less than twice the length of its shortest side. 
This remark is left to the reader.

Consider an edge~$e$ of a triangle~$\Delta''$ of~$\Tr''$ which contains at least two edges of some other triangles of~$\Tr''$.
By construction, the edge~$e$ is made up of a necessarily even number of edges of some other triangles of~$\Tr''$.
Assume that at least four edges of some other triangles of~$\Tr''$ lie in~$e$.
The length of these edges is at most one quarter of the length of~$e$.
Therefore, our initial remark shows that there exist some triangles of~$\Tr''$ whose diameter is less than half the diameter of~$\Delta''$.
Hence a contradiction with the construction of~$\Tr''$.
\end{proof}

If the edge~$e$ of a triangle~$\Delta''$ of~$\Tr''$ agrees with the union of the edges of two other triangles of~$\Tr''$, then we split~$\Delta''$ along the segment connecting the midpoint of~$e$ to the vertex of~$\Delta''$ opposite to that edge.
The new collection of triangles still satisfies the conclusion of the previous lemma.
We repeat this process to this new collection until we get a genuine triangulation~$\Tr_0$ of~$M$, where every edge is the side of exactly two adjacent triangles.


\begin{lemma} \label{lem:bilip2}
Every triangle of~$\Tr_0$ is $32$-bilipschitz homeomorphic to the equilateral flat triangle of side length~$\ell$ (the same~$\ell$ as in Lemma~\ref{lem:bilip1}).
\end{lemma}

\begin{proof}
The ($8$-bilipschitz) affine homeomorphism of Lemma~\ref{lem:bilip1} takes segments to segments and the midpoints of segments to the midpoints of the image segments.
Therefore, we can work with an equilateral flat triangle~$\Delta$ of side length~$\ell$ instead of a triangle of~$\Tr''$.
From Lemma \ref{lem:1or2}, this triangle is split into at most four triangles through the subdivision of~$\Tr''$ into~$\Tr_0$.
These subtriangles of~$\Delta$ are isometric to one of the following
\begin{enumerate}
\item the equilateral triangle of side length $\frac{\ell}{2}$;
\item the isosceles triangle with a principal angle of~$\frac{2 \pi}{3}$ and two sides of length~$\frac{\ell}{2}$;
\item the right triangle with angles $\frac{\pi}{2}$, $\frac{\pi}{3}$ and $\frac{\pi}{6}$, and side lengths $\ell$, $\frac{\sqrt{3}}{2} \ell$ and~$\frac{\ell}{2}$.
\end{enumerate}
Since each of these triangles is $4$-bilipschitz homeomorphic to~$\Delta$, we obtain the desired result.
\end{proof}

We can now conclude.

\begin{proof}[Proof of Proposition~\ref{prop:bilip}]

\forget
From Lemma \ref{lem:1or2}, every triangle of the pseudotriangulation~$\Tr''$ is split into at most three triangles and a straightforward calculation gives that the angles 
of the triangles of~$\Tr_0$ lie between $\frac{\pi}{10}$ and~$\frac{7\pi}{10}$.
Thus, the ratio between the lengths of every pair of edges of the triangulation~$\Tr_0$ is bounded by $4$.
In particular, there exists a real~$\ell$ such that the length of every edge of~$\Tr_0$ is between $\frac{1}{16}\ell$ and $16 \ell$.
Therefore, every triangle of~$\Tr_0$ is $16$-bilipschitz to the equilateral flat triangle with side length~$\ell$.
\forgotten

Denote by $M_0$ the piecewise flat surface obtained by replacing every triangle of~$\Tr_0$ by the equilateral flat triangle with side length~$\ell$.
Putting together the bilipschitz homeomorphisms of~(a) above and Lemma~\ref{lem:bilip2} gives rise to a $(1+\eps)32$-bilipschitz homeomorphism between the Riemannian surface~$M$ and the piecewise flat surface~$M_{0}$.
\end{proof}

\begin{corollary} \label{coro:same}
The area, the diameter and the diastole of $M$ and~$M_{0}$ satisfy the following inequalities
\begin{align}
K^{-2} & \leq  \frac{\area(M_0)}{\area(M)} \leq  K^{2} \label{eq:1} \\
K^{-1} & \leq  \frac{\dias(M_0)}{\dias(M)} \leq  K  \label{eq:3}
\end{align}
where $K=33$.
\end{corollary}

\begin{remark}
From Corollary~\ref{coro:same}, proving a diastolic inequality 
on~$M$ amounts to proving a diastolic inequality 
on the simplicial surface~$M_0$.
Since diastolic inequalities are scale invariant, the size of the equilateral flat triangles composing the piecewise flat surface~$M_{0}$ can arbitrarily be fixed equal to~$1$.
We will always assume this is the case when we consider simplicial surfaces.
\end{remark}

\section{Upper bounds on Cheeger's constant}\label{sec:ch}

Let $M$ be a closed Riemannian surface.
Recall that the Cheeger constant~$h(M)$ of~$M$ is defined as
$$
h(M) = \inf_{D} \frac{\length(\partial D)}{\min\{\area(D),\area(M \setminus D)\}},
$$
where the infimum is taken over all the domains~$D$ of~$M$ with smooth boundary.

\begin{proposition} \label{prop:h}
Every closed Riemannian surface~$M$ (not necessarily orientable) of genus~$g \geq 0$ satisfies the following inequality
\begin{equation} \label{eq:h}
h(M) \leq \frac{\sqrt{96 \pi (g+1)}}{\sqrt{\area(M)}}.
\end{equation}
\end{proposition}

\begin{proof}
Let $\lambda_{1}(M)$ be the first nonzero eigenvalue of the Laplacian on~$M$.
From~\cite{cheeger}, we have
$$
\lambda_{1}(M) \geq \frac{h(M)^{2}}{4}.
$$
On the other hand, from~\cite{LY}, we have
\begin{equation} \label{eq:LY}
\lambda_{1}(M) \, \area(M) \leq 24 \pi (g+1).
\end{equation}
The combination of these two inequalities yields the desired upper bound on the Cheeger constant.
\end{proof}

\begin{remark}
For orientable surfaces, the inequality~\eqref{eq:LY} still holds by replacing the constant~$24$ by~$8$, \cf~\cite{YY}.
This leads to a better upper bound on Cheeger's constant in Proposition~\ref{prop:h}, where the constant~$96$ can be replaced by~$32$.
Thus, Proposition~\ref{prop:h} improves the multiplicative constant of an inequality established in~\cite{pap} by using a different method.

Note that the dependence of the upper bound on the genus cannot be significantly improved.
Indeed, there exist closed hyperbolic surfaces of arbitrarily large genus
 with Cheeger constant bounded away from zero, \cf~\cite{bro86} for instance.
\end{remark}

An upper bound on the Cheeger constant of Riemannian two-spheres similar to~\eqref{eq:h} can be obtained using different arguments.
Although these arguments do not lead to a better multiplicative constant, they rely on a comparison between two intermediate invariants, namely the filling radius and the invariant~${\mathcal L}$, \cf~below, which is interesting in its own right.
Further, they can be generalized in higher dimension (but the final statement is not as meaningful).
For these reasons, we still include this result even though we will not use it in the rest of this paper. \\

Let $M$ be a Riemannian two-sphere.
Define
$$
{\mathcal L}(M) = \inf_{\gamma} \length(\gamma)
$$
where $\gamma$ runs over all the simple loops dividing~$M$ into two disks of area at least~$\frac{1}{4} \area(M)$.
Clearly,
$$
h(M) \leq 4 \, \frac{{\mathcal L}(M)}{\area(M)}.
$$

We show in Proposition~\ref{prop:ell} that the invariant~${\mathcal L}$ provides a lower bound on the filling radius on every Riemannian two-sphere.

\begin{definition} \label{def:fr}
Let~\mbox{$i:M \hookrightarrow L^{\infty}(M)$} be the Kuratowski distance preserving embedding defined by $i(x)(.)=d_{M}(x,.)$.
The \emph{filling radius} of~$M$, denoted by~$\FR(M)$, is defined as the infimum of the positive reals~$r$ such that the homomorphism $H_{n}(M;\Z) \longrightarrow H_{n}(U_{r};\Z)$ induced by the embedding of~$M$ into the $r$-tubular neighborhood~$U_{r}$ of its image~$i(M)$ in~$L^{\infty}(M)$ is trivial.
\end{definition}

Recall that a lower bound on the filling radius provides bounds on the area and the diameter of Riemannian two-spheres since $\FillRad(M) \leq \sqrt{\area(M)}$ from~\cite{gro83} and $\FillRad(M) \leq \frac{1}{3} \diam(M)$ from~\cite{Ka83}.

\begin{proposition} \label{prop:ell}
Let $M$ be a Riemannian two-sphere.
Then,
$$ {\mathcal L}(M) \leq 6 \, \FR(M). $$
In particular,
$$ h(M) \leq \frac{24}{\sqrt{\area(M)}}. $$
\end{proposition}

\begin{remark}
Similar lower bounds on the filling radius have been proved in~\cite{gro83} for essential manifolds, where ${\mathcal L}$ is replaced with the systole, and in~\cite{Sa04} on the two-sphere, where ${\mathcal L}$ is replaced with the shortest length of a geodesic loop.
\end{remark}

\begin{proof}
We argue by contradiction.
Pick two reals $r$ and $\eps$ such that
$$\FillRad(M) < r < r+\eps <\frac{1}{6} {\mathcal L}(M).$$
Let~\mbox{$i:M \hookrightarrow L^{\infty}(M)$} be the Kuratowski distance preserving embedding defined by $i(x)(.)=d_{M}(x,.)$.
By definition of the filling radius, there exists a map $\sigma:P \to U_{r}$ from a $3$-complex~$P$ to the $r$-tubular neighborhood $U_{r}$ of~$i(M)$ in~$L^{\infty}=L^{\infty}(M)$ such that
$\sigma_{|\partial P}: \partial P \to i(M)$ represents the fundamental class of~$i(M) \simeq M$ in $H_{2}(i(M);\Z) \simeq H_{2}(M;\Z)$.

Let us show that the map $\sigma_{|\partial P}: \partial P \to i(M)$ extends to a map \mbox{$f:P \to i(M)$}.
This would yield the desired contradiction since the fundamental class of~$i(M)$, represented by~$\sigma_{|\partial P}$, is nonzero.

Deforming $\sigma$ and subdividing $P$ if necessary, we can assume that $\sigma$ takes every edge of~$P$ to a minimizing segment and that the diameter of the $\sigma$-image of every simplex in~$P$ is less than~$\eps$.
We define $f$ to agree with~$\sigma$ on the $2$-complex~$\partial P$.

We now define $f$ on the $0$-skeleton of~$P \setminus \partial P$ by sending each vertex~$p$ of~$P \setminus \partial P$ to a point~$f(p)$ on~$i(M)$ closest to~$\sigma(p)$.
Thus, $d_{L^{\infty}}(f(p),\sigma(p)) < r$ for every vertex~$p$ of~$P$.
Since the embedding~$i$ preserves the distances, every pair of adjacent vertices $p,q \in P$ satisfies
\begin{eqnarray*}
d_{i(M)}(f(p),f(q)) & = & d_{L^{\infty}}(f(p),f(q)) \\
 & \leq & d_{L^{\infty}}(f(p),\sigma(p)) + d_{L^{\infty}}(\sigma(p),\sigma(q)) + d_{L^{\infty}}(\sigma(q),f(q)) \\
 & < & 2r+\eps=:\rho
\end{eqnarray*}

We extend $f$ to the $1$-skeleton of~$P$ by mapping each edge of~$P \setminus \partial P$ to a minimizing segment joining the images of the endpoints.
By construction, the image under~$f$ of the boundary of every $2$-simplex~$\Delta^{2}$ of~$P$ is a simple loop~$\gamma_{\Delta^{2}}$ of length less than~$3\rho$.
Since $3\rho < {\mathcal L}(M)$ for $\eps$ small enough, the simple loop~$\gamma_{\Delta^{2}}$ bounds a disk~$D_{\Delta^{2}}$ in~$i(M) \simeq M$ of area less than~$\frac{1}{4} \area(M)$.
We extend $f$ to the $2$-skeleton of~$P$ by mapping each $2$-simplex~$\Delta^{2}$ to the disk~$D_{\Delta^{2}}$.
Thus, the restriction $\partial \Delta^{3} \to i(M)$ of $f$ to the boundary~$\partial \Delta^{3}$ of a $3$-simplex~$\Delta^{3}$ of~$P$ is a degree zero map between two spheres since the area of its image is less than the area of~$i(M) \simeq M$.
Therefore, $f$ extends to each simplex of~$P$ and gives rise to an extension $f:P \to i(M)$ of~$\sigma_{|\partial P}$.
\end{proof}

\section{Isoperimetric inequalities on simplicial surfaces} \label{sec:ch2}

We will need the following remark.
Let $\eps>0$.
Every piecewise flat metric with conical singularities on a surface can be smoothed out at its singularities into a Riemannian metric $(1+\eps)$-bilipschitz homeomorphic to it.
Therefore, Proposition~\ref{prop:h} still holds true for closed simplicial surfaces. \\

We prove now that the proposition~\ref{prop:h} still holds in a simplicial setting.

\begin{proposition} \label{prop:split}
Let $M_{0}$ be a closed simplicial surface (not necessarily orientable) of genus~$g \geq 0$ formed of~$N$ triangles.
The surface~$M_{0}$ decomposes into two simplicial domains, $D_{1}$ and~$D_{2}$, with disjoint interiors, satisfying the following inequality for $i=1,2$
\begin{equation} \label{eq:split}
\length(\partial D_{i}) \leq C_{0} \sqrt{g+1} \, \frac{\min\{\area(D_{1}),\area(D_{2})\}}{\sqrt{\area(M_{0})}},
\end{equation}
where $C_{0}=15 \sqrt{96 \pi}$.
\end{proposition}

\begin{proof}
Assume that $N \geq N_*$ where $N_*:=10 \cdot 96 \pi (g+1)$.
Fix $\eps \in]0,\frac{1}{50}[$.
Since $M_{0}$ is composed of~$N$ equilateral flat triangles (of side length~$1$), the area of~$M_{0}$ is equal to
\begin{equation} \label{eq:N}
\area(M_{0}) = \frac{\sqrt{3}}{4} N.
\end{equation}
From the remark at the beginning of this section, the surface~$M_{0}$ decomposes into two (not necessarily simplicial) domains $D_{1}$ and~$D_{2}$ with common boundary~$\delta$ such that
\begin{equation} \label{eq:prov}
\length(\delta) < (1+\eps) \, \sqrt{96 \pi} \, \sqrt{g+1} \, \frac{\min\{\area(D_{1}),\area(D_{2})\}}{\sqrt{\area(M_{0})}}.
\end{equation}
Combined with the formula~\eqref{eq:N} and the bound~$N \geq N_*$, this estimate yields
\begin{equation} \label{eq:1/2}
\length(\delta) < \frac{1}{2} \, \area(D_{j}) \text{ for } j=1,2.
\end{equation}
We want to deform~$\delta$ into a one-cycle lying in the one-skeleton of~$M_{0}$ while controlling its length and the area of the two domains it bounds.

Without loss of generality, we can suppose that
\begin{enumerate}
\item $\delta$ is composed of finitely many simple loops;
\item each of these loops meets the edges of the triangulation of~$M_0$;
\item each of these loops intersects the edges of the triangulation at finitely many points;
\item none of these loops passes through the midpoints of the edges of the triangulation.
\end{enumerate}

In this situation, the one-cycle~$\delta$ decomposes into a finite union of subarcs~$\tau_i$, where $i \in I$, such that each arc~$\tau_i$ is contained into a triangle~$\Delta$ of the triangulation of~$M_0$ with its endpoints lying in the boundary~$\partial \Delta$ of~$\Delta$.
For every $i \in I$, denote by~$\ell_i$ the length of~$\tau_i$.
We have
$$
\length(\delta) = \sum_{i \in I} \ell_i.
$$

For every triangle~$\Delta$ of the triangulation of~$M_0$, we apply the curve-shortening process defined in~\cite[\S2.2]{Sa04} to the finite collection of arcs~$\tau_i$ lying in the convex domain~$\Delta$.
Through this curve-shortening process, the arcs~$\tau_i$ remain disjoint, do not self-intersect and converge to the segments~$s_i$ of ~$\Delta$ with the same endpoints.
(Strictly speaking the convergence property has been stated for bumpy metrics but it still holds in the flat case since every couple of points of~$\Delta$ are connected by a unique geodesic.)
Now, we extend this process by deforming simultaneously every segment~$s_i$ into a vertex or an edge of the triangulation through a family of segments whose endpoints move along the sides of~$\Delta$ to the closest vertices of~$\Delta$ (see figure~\ref{deform}).

\begin{figure}[h]
\begin{center}
\AffixLabels{\centerline{\epsfig{file =deform.pdf,width=12cm,angle=0}}}
\end{center}
\caption{} \label{deform}
\end{figure}

This deformation~$(\delta_t)$ of~$\delta$ can be performed so that the loops forming the one-cycles~$\delta_t$ remain simple, except possibly for the final loop~$\delta'=\delta_\infty$.
This defines our deformation process.

The final one-cycle~$\delta'$ lies in the one-skeleton of~$M_0$.
Therefore, the two domains bounded by~$\delta_t$ converge to two simplicial domains $D'_1$ and $D'_2$ of~$M_0$ whose common boundary lies in~$\delta'$.

\begin{lemma} \label{lem:gamma'}
We have
$$
\length(\delta') \leq 2 \, \length(\delta).
$$
\end{lemma}

\begin{proof}
Let $i \in I$.
By definition, the arc~$\tau_{i}$, of length~$\ell_{i}$, lies in a triangle~$\Delta$ of the triangulation of~$M_{0}$ and $\length(s_i)\leq \ell_i$.
Through the deformation process, the segment~$s_i$ deforms into either a vertex or an edge of~$\Delta$.
The latter case occurs when the two endpoints of~$s_{i}$ are at distance less than~$\frac{1}{2}$ from two distinct vertices of~$\Delta$.
Thus, the segment~$s_{i}$ deforms into a path of length at most twice the length of~$s_{i}$.
(The limit case occurs when $s_{i}$ joins the midpoints of two edges of~$\Delta$.)
Summing up over all the $i \in I$ yields the desired inequality.
\end{proof}

\begin{lemma} \label{lem:DD}
For $j=1,2$, we have
\begin{equation} \label{eq:DD}
\area(D'_j) \geq \frac{1}{2} \, \area(D_{j}).
\end{equation}
\end{lemma}

\begin{proof}
The intersection $D_{j} \cap \Delta$ between $D_{j}$ and a triangle~$\Delta$ of the triangulation of~$M_{0}$ is formed of domains bounded by the arcs~$\tau_{i}$ and the edges of~$\Delta$.
Consider $i \in I$ such that $\tau_{i}$ is an arc of~$\Delta$.

Suppose that the length~$\ell_{i}$ of~$\tau_{i}$ is less than $\frac{1}{2}$.

If the endpoints of~$\tau_{i}$ lie in the same edge~$e$ of~$\Delta$, the arc~$\tau_{i}$ deforms either to this edge or to an endpoint of this edge through the deformation process previously defined (see figure~\ref{deform}).
And so does the region of~$\Delta$ bounded by~$\tau_{i}$ and the edge~$e$.

If the endpoints of~$\tau_{i}$ lie in two adjacent edges $e_{1}$ and~$e_{2}$ of~$\Delta$, the endpoints of~$\tau_{i}$ move to the same vertex~$v$ of~$\Delta$ through the deformation process (recall that the length of~$\tau_{i}$ is less than~$\frac{1}{2}$).
Thus, the arc~$\tau_{i}$ deforms into this vertex and the region of~$\Delta$ bounded by~$\tau_{i}$ and the edges $e_{1}$ and~$e_{2}$ retracts onto the vertex~$v$.

In both cases, the area of this region of~$\Delta$ bounded by~$\tau_{i}$ and either the edge~$e$ or the edges $e_{1}$ and~$e_{2}$ is less or equal to $\frac{3}{2 \pi} \ell_{i}^{2}$.
This isoperimetric inequality comes from the standard isoperimetric inequality on the plane combined with a symmetry argument.

Suppose now that the length~$\ell_{i}$ of~$\tau_{i}$ is at least $\frac{1}{2}$.
In this case, a rougher estimate holds.
Specifically, the area of each of the two domains of~$\Delta$ bounded by~$\tau_{i}$ is less than the area of~$\Delta$, \ie, $\frac{\sqrt{3}}{4}$. \\

Therefore, by substracting up the area change of~$D_{j} \cap \Delta$ through the deformation process for every triangle of the triangulation, we obtain
\begin{eqnarray}
\area(D'_{j}) & \geq & \area(D_{j}) - \sum_{i \in I^{-}} \frac{3}{2 \pi} \ell_{i}^{2} - \sum_{i \in I^{+}} \frac{\sqrt{3}}{4} \nonumber \\
              & \geq & \area(D_{j}) - \sum_{i \in I^{-}} \frac{3}{4 \pi} \ell_{i} - \sum_{i \in I^{+}} \frac{\sqrt{3}}{2} \ell_{i} \nonumber \\
              & \geq & \area(D_{j}) - \frac{\sqrt{3}}{2} \left( \sum_{i \in I} \ell_{i} \right) \label{eq:D'}
\end{eqnarray}
where the first sum is over all the indices~$i$ for which $\ell_{i} < \frac{1}{2}$ and the second over those for which $\ell_{i} \geq \frac{1}{2}$.

From the inequality~\eqref{eq:1/2}, namely
\begin{equation*}
\length(\delta) = \sum_{i \in I} \ell_{i} < \frac{1}{2} \area(D_{j}),
\end{equation*}
we derive
\begin{equation*}
\area(D'_{j}) \geq \frac{1}{2} \area(D_{j}).
\end{equation*}
\end{proof}

From the inequality~\eqref{eq:prov} and the lemmas~\ref{lem:gamma'} and~\ref{lem:DD}, we deduce
$$
\length(\delta') < 4\, (1+\eps) \, \sqrt{96 \pi} \, \sqrt{g+1} \, \frac{\min\{\area(D'_{1}),\area(D'_{2})\}}{\sqrt{\area(M_{0})}}.
$$

When $N \leq N_*$, we can take for~$D_{1}$ any triangle of~$M_{0}$ and for~$D_{2}$ the union of the remaining triangles.
Thus,
$$
\length(\partial D_{i}) = 3 \quad \text{ and } \quad \min \{\area(D_{1}),\area(D_{2}) \} =\frac{\sqrt{3}}{4}.
$$
Using the formula~\eqref{eq:N} and the bound $N \leq N_*$, one can check that the inequality~\eqref{eq:split} is satisfies in this case too.
Hence the conclusion.
\end{proof}

\section{Diastole comparison} \label{sec:max}

Let~$M_0$ be a simplicial closed surface decomposed into two simplicial domains $D_1$ and~$D_2$ with disjoint interiors.
Denote by~$\delta$ the common boundary of the domains $D_1$ and~$D_2$.
Consider the abstract simplicial cones over the connected components of~$\partial D_i$ composed of triangles with basis the edges of~$\delta=\partial D_i$.
Denote by $M_i$ the simplicial closed surface obtained by attaching these simplicial cones to~$D_i$ along~$\partial D_i$.

The following estimate on the (technical) diastole, \cf~Definition~\ref{def:dias}, follows from a cut-and-paste argument.

\begin{proposition} \label{prop:max}
We have
\begin{equation}
\dias'(M_0) \leq \max_i \dias'(M_i) + \length(\delta).
\end{equation}
\end{proposition}

\begin{proof}
Fix $\epsilon >0$. By definition, for $i=1,2$, there exists a one-parameter family $(z^i_t)_{0\leq t\leq 1}$ of one-cycles on $M_i$, with coefficients in~$\kk$, which satisfies the conditions (D.1-5) of the definition of the (technical) diastole, \cf~Definition~\ref{def:dias}, such that

\forget
\begin{enumerate}
\item $z^i_t$ starts and ends at null-currents;
\item $z^i_t$ induces a nontrivial class $[z_i]$ in
$\pi_1(\mathcal{Z}_1(M_i,\Z/2\Z),\{0\})$;
\item and
\forgotten

$$
\dias'(M_i) \leq \sup_{0 \leq t \leq 1} \M(z^i_t) < \dias'(M_i)+\epsilon.
$$

Denote by $(\gamma_{j,t}^{i})$ the homotopies of the loops defining the family~$(z_{t}^{i})$, \cf~item~(D.3) of Definition~\ref{def:dias}.
By slightly perturbing the family $(z^i_t)$ and refining the subdivision $(t_p)$ if necessary, we can assume that the loops~$\gamma_{j,t}^{i}$ are disjoint or transverse to~$\partial D_i$ for every $t \neq t_p$, and that they have at most finitely many tangent points with~$\partial D_i$ for every~$t_{p}$, where $p=1,\cdots,k$.

\forget
The map taking a one-cycle~$z$ of~$M$ to the relative one-cycle~$z \llcorner D_i$ of $(D_i,\partial D_i)$ then to the zero-cycle~$\partial(z \llcorner D_i)$ of~$\partial D_i$ induces the first row of the following commutative diagram.
In this diagram, the second row is induced by the inclusions $\partial D_i \subset D_i \subset M$ and the vertical maps are natural isomorphisms.
See \cite{almgren}, \cite{alm65}.

\begin{equation*}
\begin{CD}
\pi_1(\mathcal{Z}_1(M),\{0\}) @>>> \pi_1(\mathcal{Z}_1(D_i,\partial D_i),\{0\}) @>>> \pi_1(\mathcal{Z}_0(\partial D_i),\{0\}) \\
@VVV @VVV @VVV \\
H_2(M;\Z/2\Z) @>>> H_2(D_i,\partial D_i;\Z/2\Z) @>>> H_1(\partial D_i;\Z/2\Z) \\
\end{CD}
\end{equation*}
\forgotten

The boundary~$\partial (z^i_t \llcorner D_i)$ of the restriction of~$z^i_t$ to the domain~$D_i$ defines a family of zero-cycles on~$\partial D_i$ with disjoint supports composed of finitely many points.
This family of one-cycles sweeps out~$\partial D_i$, that is, it starts and ends at null one-cycles, and induces a generator of~$\pi_1(\mathcal{Z}_0(\partial D_i;\kk),\{0\})$.
It also induces a family of one-chains~$\alpha^i_t$ on~$\partial D_i$ such that $\alpha^i_0 = 0$, $\alpha^i_1 = \pm \partial D_i$ and $\partial \alpha^i_t = \partial (z^i_t \llcorner D_i)$.

In a more constructive way, this family of one-chains on~$\partial D_{i}$ can be described as follows.
The restriction $z^i_t \llcorner D_i$ of~$z^i_t$ to~$D_i$, with $t_p \leq t \leq t_{p+1}$, decomposes as a finite sum of loops lying in~$D_i$ plus a finite sum of arcs~$c_{j,t}^i$ lying in~$D_i$ with endpoints on~$\partial D_{i}$.
As $t$ runs over $[t_p,t_{p+1}]$, the basepoints and the endpoints of the arcs~$c_{j,t}^i$ describe some arcs on~$\partial D_i$, \cf~\cite[\S14]{rham} for a more general construction.
The arcs described by the basepoints (resp. endpoints) of the~$c_{j,t}^i$'s inherit a negative (resp. positive) orientation.
Note that these arcs do not overlap because of the condition~(D.5).
Putting together the sum of these arcs on~$\partial D_i$ gives rise to a family of one-chains~$\alpha_t^i$ on~$\partial D_i$, starting at the null one-chain and ending at the one-cycle~$\pm \partial D_i$, such that $\partial \alpha^i_t = \partial (z^i_t \llcorner D_i)$.

Consider the family of one-cycles $\tilde{z}^i_t=(z^i_t \llcorner D_i) - \alpha^i_t$ lying in~$D_i$.
By construction, this family starts at the null one-cycle, ends at the one-cycle~$\pm \partial D_i$ and induces a generator of~$\pi_1(\mathcal{Z}_1(D_i,\partial D_i;\kk),\{0\})$.
Furthermore, its mass is bounded as follows
$$
\M(\tilde{z}^i_t) \leq \M(z^i_t)+\length(\delta).
$$

\forget
and $\tilde{z}^i_1=\pm \gamma$.
We can suppose that both $\tilde{z}^i_t$, $i=1,2$ are such that $\tilde{z}^i_1=\gamma$. If not consider $-\tilde{z}^i_t$ instead of $\tilde{z}^i_t$.
\forgotten

We can now define a new one-parameter family~$(\tilde{z}_t)_{0\leq t \leq 1}$ of one-cycles on~$M_0$ by concatenating  $\tilde{z}^1_t$ and~$\tilde{z}^2_t$ as follows
$$
\tilde{z}_t = \left\{\begin{array}{ll}
\tilde{z}^0_{2t} & \text{if } t \in [0,1/2], \\
\\
- \tilde{z}^1_{2-2t} & \text{if } t \in [1/2,1]. \\
\end{array}\right .
$$
By construction, this new family satisfies the conditions (D.1-4) of the definition of the (technical) diastole, \cf~Definition~\ref{def:dias}, and
\begin{equation} \label{eq:eps}
\M(\tilde{z}_t) < \max_i \dias'(M_i)+ \length(\delta) +\epsilon.
\end{equation}
It does not necessarily satisfy the condition~(D.5) though.
Indeed, the supports of the one-cycles~$\tilde{z}_{t}$ may have a nonempty intersection on~$\partial D_{i}$.

By slightly perturbing the family~$(\tilde{z}_t)$ in the neighborhood of~$\partial D_{i}$, it is possible to obtain a family of one-cycles satisfying the conditions~(D.1-5) and the inequality~\eqref{eq:eps}

\forget
\begin{enumerate}
\item $\tilde{z}_t$ starts and ends at null-currents;
\item $\tilde{z}_t$ induces a nontrivial class~$[\tilde{z}]$ in $\pi_1(\mathcal{Z}_1(M_0,\Z/2\Z),\{0\})$;
\item and
$$
\M(\tilde{z}_t) < \max_i \dias'(M_i)+ \length(\delta) +\epsilon.
$$
\end{enumerate}
\forgotten
Therefore, $\dias'(M_0) \leq \max_i \dias'(M_i)+ \length(\delta) +\epsilon$ for any positive $\epsilon $.
Hence the proposition.
\end{proof}

\section{Proof of the diastolic inequality}\label{sec:proof}

We can now prove the main diastolic inequality, \cf~Theorem~\ref{theo:A}.
We will use the notations of the previous sections.

From the discussion in Section~\ref{sec:equil}, \cf~Corollary~\ref{coro:same}, it is enough to show that
\begin{equation} \label{eq:C/K}
\dias'(M_{0}) \leq \frac{C}{K^{2}} \, \sqrt{g+1} \, \sqrt{\area(M_{0})}
\end{equation}
for every simplicial closed surface~$M_0$ (not necessarily orientable) of genus~$g$.
We will argue by induction both on the number~$N$ of triangles composing~$M_{0}$ and on its genus. From now on we suppose that Theorem~\ref{theo:A} is proved for surfaces of genus less than ~$g$. If $M_0$ is a $2$-sphere, the proof proceeds by induction only on the number of simplices.   
\medskip

Decompose $M_0$ into two simplicial domains $D_1$ and~$D_2$ as in Proposition~\ref{prop:split}.
Consider the simplicial closed surfaces~$M_i$ defined in Section~\ref{sec:max} by attaching simplicial cones along~$\partial D_i$.
By construction, the genus of~$M_i$ is at most~$g$.
Since~$\partial D_i$ is made up of $\length(\partial D_i)$ edges, the area of the cones glued to~$D_i$ is equal to $\frac{\sqrt{3}}{4} \, \length(\partial D_i)$.
Hence the formula
\begin{equation} \label{eq:area}
\area(M_i) = \area(D_{i}) + \frac{\sqrt{3}}{4} \length(\partial D_{i}).
\end{equation}

Set $N_0:= \sqrt{3} \, C_0^2 \, (g+1)=225 \sqrt{3} \cdot 96 \pi (g+1).$

\begin{lemma} \label{lem:fewer}
For $N \geq N_{0}$, the simplicial closed surfaces~$M_{i}$ are composed of fewer triangles than~$M_{0}$.
\end{lemma}

\begin{proof}
Recall that $\area(M_{0}) = \frac{\sqrt{3}}{4} N$.
Thus, for $N \geq N_{0}$, the quantity $\frac{\sqrt{3}}{4} C_{0} \frac{\sqrt{g+1}}{\sqrt{\area(M_{0})}}$ is less than~$1/2$.
From the formula~\eqref{eq:area} and the bound~\eqref{eq:split}, we deduce that
\begin{eqnarray*}
\area(M_i) & < & \area(D_{i}) + \min\{\area(D_{1}),\area(D_{2})\} \\
& < & \area(D_{1}) + \area(D_{2}).
\end{eqnarray*}
Hence, $\area(M_i) <  \area(M_{0})$.
\end{proof}

Renumbering the $M_i$'s if necessary, we can assume that
$$
\dias'(M_1) \geq \dias'(M_2).
$$

Let $\alpha$ be a real such that $\alpha \, \area(M_{0}) = \min\{\area(D_{1}),\area(D_{2})\}$.
Note that $0 < \alpha \leq \frac{1}{2}$.
The inequality~\eqref{eq:split} can be written
\begin{equation} \label{eq:alpha}
\length(\delta) \leq C_{0} \, \alpha \, \sqrt{g+1} \sqrt{\area(M_{0})},
\end{equation}
where $\delta = \partial D_{i}$.

We have two cases to consider. \\

{\it First case:} Suppose that the number~$N$ of triangles of~$M_0$ is at most~$N_0$, where $N_{0}$ is defined in Lemma~\ref{lem:fewer}.
Fix an orientation on the surface~$M_{0}$.
This orientation induces an orientation on all the triangles of~$M_0$ and on their boundaries.
Of course, if the surface is nonorientable, there is nothing to do.
Consider the boundaries of the triangles of~$M_{0}$ as one-cycles with coefficients in~$\kk$.
The sum of these boundaries, viewed as one-cycles, vanishes.
Given a triangle~$\Delta$ of~$M_0$, there exists a homotopy from the boundary of~$\Delta$, with the induced orientation,
to the center of~$\Delta$ composed of disjoint loops.
The sum of these homotopies over all the triangles of~$M_0$ defines a one-parameter family of one-cycles~$(z_t)_t$ on~$M_{0}$, which starts and ends at the null one-cycle.
Furthermore, it induces a generator of~$\pi_1(\mathcal{Z}_1(M_0;\kk),\{0\}) \simeq H_2(M_0;\kk)$ and satisfies the technical conditions (D.3-5) of Definition~\ref{def:dias}.
Therefore,
$$
\dias'(M_0) \leq \sup_t \M(z_t) \leq 3 N.
$$
Since $\area(M_0) = \frac{\sqrt{3}}{4} N$ and $N \leq N_0$, we obtain
\begin{eqnarray*}
\dias'(M_0) & \leq & 3 \sqrt{N} \sqrt{N_{0}} \\
 & \leq & 2 \cdot 3^{3/4} \, \sqrt{N_0} \sqrt{\area(M_0)} \\
                    & \leq & 6 \, C_{0} \, \sqrt{g+1} \, \sqrt{\area(M_0)}.
\end{eqnarray*}

\bigskip

{\it Second case:} Suppose that $N$ is greater than~$N_0$.
We can assume that
$$
\frac{\dias'(M_1)}{\sqrt{\area(M_1)}} \leq \frac{\dias'(M_0)}{\sqrt{\area(M_0)}}
$$
otherwise the diastolic inequality~\eqref{eq:C/K} follows from Lemma~\ref{lem:fewer} by induction on~$g$ and~$N$.
That is,
$$
\dias'(M_{1}) \leq \lambda \, \dias'(M_{0}),
$$
where
\begin{equation} \label{eq:lambda}
\lambda := \sqrt{\frac{\area(M_1)}{\area(M_0)}}.
\end{equation}
Note that $\lambda <1$ from Lemma~\ref{lem:fewer}.
From Proposition~\ref{prop:max} and since $\dias'(M_1) = \max_i \dias'(M_{i})$, we deduce
\begin{eqnarray*}
\dias'(M_{0}) & \leq & \frac{1}{1-\lambda} \, \length(\delta).
\end{eqnarray*}
From the inequality~\eqref{eq:alpha}, we obtain
\begin{equation} \label{eq:E}
\dias'(M_{0}) \leq C_{0} \, \frac{\alpha}{1-\lambda} \, \sqrt{g+1} \, \sqrt{\area(M_{0})}.
\end{equation}

\begin{lemma}
We have $ \displaystyle \frac{\alpha}{1-\lambda} \leq 4$.
\end{lemma}

\begin{proof}
\forget
Let $\kappa= \frac{\sqrt{3}}{4} \, C_{0} \, \frac{\sqrt{g+1}}{\sqrt{\area(M_{0})}}$.
Since $\area(M_{0}) = \frac{\sqrt{3}}{4} N$ and $N \geq N_{0}$, where $N_{0}$ is defined in Lemma~\ref{lem:fewer}, we have $\kappa \leq \frac{1}{3}$.
We will consider two cases.
\forgotten

Since $N \geq N_0$, the quantity $\frac{\sqrt{3}}{4} C_{0} \frac{\sqrt{g+1}}{\sqrt{\area(M_{0})}}$ is less than~$\frac{1}{2}$ and
$$ \area(M_i) \leq \area(D_i) + \frac{1}{2} \min\{\area(D_1),\area(D_2)\} $$
from the formula~\eqref{eq:area} and the bound~\eqref{eq:split}.

If $\area(D_{1}) \leq \area(D_{2})$, then $\area(D_{1}) = \alpha \, \area(M_{0})$.
In this case, we deduce that $\lambda \leq \sqrt{\frac{3}{2} \alpha}$ from the definition~\eqref{eq:lambda} of~$\lambda$.
Since $\alpha \leq \frac{1}{2}$, we obtain
\begin{equation*}
\frac{\alpha}{1-\lambda} \leq  \frac{\alpha}{1-\sqrt{\frac{3}{2} \alpha}} \leq  2 + \sqrt{3}.
\end{equation*}

Otherwise, $\area(D_{1}) = (1-\alpha) \, \area(M_{0})$.
As previously, we derive, in this case, that $\lambda \leq \sqrt{1-\frac{\alpha}{2}}$.
Since $\alpha > 0$, we get
\begin{equation*}
\frac{\alpha}{1-\lambda} \leq 2\left( 1 + \sqrt{1-\frac{\alpha}{2}} \right) \leq  4.
\end{equation*}
\end{proof}

This lemma with the inequality~\eqref{eq:E} implies that
\begin{equation*}
\dias'(M_{0}) \leq 4 \, C_{0} \, \sqrt{g+1} \, \sqrt{\area(M_{0})}.
\end{equation*}

\bigskip

Therefore, the desired diastolic inequality~\eqref{eq:C/K} holds in both cases.

\begin{remark} \label{rem:g}
Let $(z_t)$ be a family of one-cycles which sweeps out a closed Riemannian surface~$M$, \cf~(D.1-2).
From the constructions of~\cite{almgren}, there exists $t_0$ such that $z_{t_0}$ decomposes~$M$ into two disjoint domains of area at least $\frac{1}{4} \area(M)$.
Hence,
$$\dias(M) \geq \M(z_{t_0}) \geq \frac{1}{4} h(M) \area(M).$$
The closed hyperbolic surfaces of arbitrarily large genus with Cheeger constant bounded away from zero constructed in~\cite{bro86} provide examples of surfaces with $\area(M) \simeq g$ and $\dias(M) \gtrsim g$.
Thus, the dependence of the inequality~\eqref{eq:A} on the genus is optimal.
\end{remark}

\section{Comparison of Riemannian invariants on the two-sphere} \label{sec:comp}

In this section, we clarify the relationships between several Riemannian invariants on the two-sphere by comparing them up to an equivalence relation.

\begin{definition}
Two Riemannian invariants $I_1$ and~$I_2$ are said to be \emph{essentially equivalent} on a manifold~$M$ if there exist two positive constants $C$ and~$C'$ such that for every Riemannian metric~$\G$ on~$M$, we have
$$
C \, I_1(\G) \leq I_2(\G) \leq C' \, I_1(\G).
$$
We denote by~$I_1(M) \simeq I_2(M)$ this equivalence relation.
\end{definition}

Together with the filling radius introduced in Definition~\ref{def:fr}, we will consider the following Riemannian invariants.

\begin{definition} \label{def:L}
Let $M$ be a closed $n$-dimensional Riemannian manifold.
\begin{itemize}


\item The \emph{$1$-hypersphericity}, denoted by~$\hs(M)$, is the supremum of the positive reals~$R$ so that there is a contracting map of degree $1$ from $M$ to the $n$-sphere of radius $R$.

\medskip

\item The \emph{Uryson $1$-width}, denoted by~$\w(M)$, is the infimum of the positive reals~$W$ so that there is a continuous map from $M$ to a $1$-dimensional polyhedron whose fibers have diameter less than $W$.

\medskip

\item
Let us define the functional $\mu$ on the one-cycle space~$\mathcal{Z}_1(M,\Z)$ of a Riemannian two-sphere~$M$ as
$$
\mu(z) = \sup \{ \mass(z_i) \mid z_i \text{ is a connected component of } z \}.
$$

We define a critical value of the functional $\mu$ by a global minimax principle
$$
L(M)=\inf_{(z_{t})} \sup_{0 \leq t \leq 1} \mu(z_t)
$$
where  $(z_{t})$ runs over the families of one-cycles satisfying the following conditions:

\begin{enumerate}
\item[(C.1)]  $z_t$ starts and ends at null one-cycles;
\item[(C.2)] $z_t$ induces a generator of $H_2(M;\Z)$.
\end{enumerate}

\end{itemize}
\end{definition}

\forget

\bigskip

In \cite{Gr88} M. Gromov proves that
$$
\hs(M)\leq \w(M)
$$
for all Riemannian surface $M$. Using this inequality and a result of L. Guth \cite[Theorem 0.9]{Guth05} which states that every closed oriented Riemannian surface $M$ of genus $g$ satisfies
$$
\w(M)\leq (200g+12)\hs(M),
$$
 we deduce the

\begin{proposition}\label{Gu}
For every closed oriented surface $\Sigma$
$$
\hs(\Sigma) \simeq \w(\Sigma).
$$
\end{proposition}

Also note that for every Riemannian manifold $M$ of dimension $n$
$$
\FR(S^n) \, \hs(M)\leq \FR(M)
$$
as the filling radius decreases under contracting maps of degree one (see \cite{gro83}).

\begin{remark}
From \cite{Ka83} we know the exact value of the filling radius of the $n$-sphere:
$$
\FR(S^n)=\frac{1}{2}\arccos\left(\frac{-1}{n+1}\right).
$$
\end{remark}

\bigskip

We now turn our attention on $2$-spheres. By the above proposition the $1$-hypersphericity and the Uryson $1$-width are essentially equivalent on the two-sphere. However the diastole and the filling radius are \emph{not} essentially equivalent, see~\cite[Theorem~1.6]{Sa04}. In order to compare the $1$-hypersphericity and the Uryson $1$-width with the filling radius we introduce a new Riemannian invariant - a weaker version  of the diastole - as follows.

 Furthermore it follows from \cite[Theorem 0.3]{Guth05} that
$$
L(M)\leq 120 \hs(M)
$$
for every Riemannian two-sphere~$M$.

\medskip

As $\FR(M)\leq *L(M)$ (WHY) we deduce the

\forgotten

In the following proposition, we compare these Riemannian invariants on the two-sphere.

\begin{proposition} \label{prop:comp}
The hypersphericity, the Uryson $1$-width, the filling radius and $L$ are essentially equivalent on the two-sphere~$S^2$, that is
$$
\w(S^2)\simeq\hs(S^2)\simeq\FR(S^2)\simeq L(S^2).
$$
\end{proposition}

\begin{proof}
Since the filling radius of every closed Riemannian $n$-manifold decreases under contracting maps of degree one, \cf~\cite{gro83}, we have
\begin{equation} \label{eq:a}
\FR(S^2,\can) \, \hs(M) \leq \FR(M)
\end{equation}
where the value of the filling radius of the standard spheres has been computed in~\cite{Ka83}.
Here, $\FR(S^2,\can) = \frac{1}{2}\arccos\left(-\frac{1}{3}\right)$.

The inequality
\begin{equation} \label{eq:b}
\w(M) \leq \hs(M)
\end{equation}
 holds for every closed Riemannian manifold, \cf~\cite{Gr88}.

The inequalities
\begin{equation} \label{eq:c}
\frac{L(M)}{120} \leq \hs(M) \leq \frac{2}{\pi} \w(M)
\end{equation}
have been established in~\cite[Theorem~0.3 and p.~1058]{Guth05}.

Next we will prove that there exists a constant~$C$
\begin{equation} \label{eq:d}
\FR(M) \leq C \, L(M).
\end{equation}
First of all observe that if $(z_{t})$ is a family of one-cycles on~$M$ and if $t_{1}$ and~$t_{2}$ are close enough, the one-cycle $z=z_{t_{2}}-z_{t_{1}}$ bounds a $2$-chain of small mass on~$M$.
Indeed, by definition of the flat norm, \cf~Section~\ref{sec:dias}, and since the one-cycle~$z$ has a small flat norm, there exists a $2$-chain~$A$ such that both $\M(z-\partial A)$ and~$\M(A)$ are small.
By the isoperimetric inequality on~$M$, there exists a $2$-chain~$B$ of small mass bounding the one-cycle~$z-\partial A$.
Thus, the $2$-chain~$A+B$ bounds~$z$ and has a small mass.
Now, using the arguments of~\cite[p.~133]{gro83}, one can show that, for every $\varepsilon > 0$, there exists a $1$-Lipschitz map $\partial X \to M$ of degree one, where $X$ is a $3$-dimensional pseudomanifold endowed with a metric length structure such that $d(x,\partial X) \leq C \, L(M) + \varepsilon$.
We immediately derive the inequality~\eqref{eq:d}.

Putting together the inequalities \eqref{eq:a}, \eqref{eq:b}, \eqref{eq:c} and \eqref{eq:d} yields the desired relations.
\end{proof}

\begin{remark}
We have $L(M) \leq \dias(M)$ for every Riemannian two-sphere~$M$ but the diastole is \emph{not} essentially equivalent to any of the Riemannian invariant of Proposition~\ref{prop:comp} on the two-sphere from~\cite[Theorem~1.6]{Sa04}.
In the appendix, we bound from below the filling radius of the Riemannian two-spheres in terms of the invariant~$L$ using different methods.
\end{remark}

\renewcommand{\theequation}{A.\arabic{equation}}
\setcounter{equation}{0}

\renewcommand{\thesection}{\Alph{section}}
\setcounter{section}{1}

\theoremstyle{theorema}
\newtheorem{theorema}{Theorem}[section]
\newtheorem{propositiona}[theorema]{Proposition}
\newtheorem{corollarya}[theorema]{Corollary}
\newtheorem{lemmaa}[theorema]{Lemma}

\theoremstyle{definition}
\newtheorem{definitiona}[theorema]{Definition}
\newtheorem{examplea}[theorema]{Example}
\newtheorem{remarka}[theorema]{Remark}
\newtheorem{conjecturea}[theorema]{Conjecture}

\section*{Appendix}

In this appendix, we  prove the following result which first appeared in~\cite[\S2.4.2]{Sa01} and was not published in a journal.

\begin{propositiona} \label{prop:L18}
Let $M$ be a Riemannian two-sphere, then
$$L(M)\leq 18 \, \FillRad(M). $$
\end{propositiona}

\begin{remarka}
The multiplicative constant in this inequality is better than the one obtained in the proof of Proposition~\ref{prop:comp}.
Furthermore, it follows from a more general result, namely Proposition~\ref{prop:L}.
\end{remarka}

From the bounds $\FillRad(M) \leq \sqrt{\area(M)}$, \cf~\cite{gro83}, and $\FillRad(M) \leq \frac{1}{3} \diam(M)$, \cf~\cite{Ka83}, we immediately deduce the following corollary.

\begin{corollarya}
Let $M$ be a Riemannian two-sphere, then
\begin{align}
L(M) & \leq 18 \, \sqrt{\area(M)}  \label{eq:La} \\
L(M) & \leq 6 \, \diam(M).\label{eq:Ld}
\end{align}
\end{corollarya}

\forget

\begin{remarka}
Let $M$ be a Riemannian two-sphere. From \cite{Guth05} we know that $L(M)\leq 120 \hs(M)$. As $\FR(S^2) \, \hs(M)\leq \FR(M)$ we get
$$
L(M) \leq \frac{120}{\FR(S^2)}\FR(M)\approx 125,6\FR(M).
$$
So our constant in the Theorem \ref{theo:L} is better than the one we can deduce from Guth's work.
\end{remarka}

\forgotten

In order to prove Proposition~\ref{prop:L18}, we need to extend the definition of the invariant~$L$, \cf~Definition~\ref{def:L}, as follows.

Let $N \subset M$ be a connected domain whose boundary $\partial N$ is a finite union of closed geodesics $c_i$ with $i \in I$.

Let us consider the one-parameter families of one-cycles $(z_t)_{0 \leq t \leq 1}$ on $N$ which satisfy the following conditions.
\begin{enumerate}
\item[(C'.1)] $z_0 = \cup_{i \in I_0} c_i$ and $z_1=\cup_{i \in I_1} c_i$ for some partition of $I=I_0 \amalg I_1$
\item[(C'.2)] $z_t$ induces a generator of $H_2(N,\partial N;\Z)$.
\end{enumerate}
If $I_0$ or $I_1$ is empty, the corresponding one-cycles are reduced to the null one-cycle. \\

We define a critical value of the functional $\mu$ by a global minimax principle
$$ L(N)=\inf_{(z_{t})} \sup_{0 \leq t \leq 1} \mu(z_t) $$
where  $(z_{t})$ runs over the families of one-cycles satisfying the conditions~\mbox{(C'.1-2)} above.
\bigskip

We will need a few more definitions.
Every nontrivial simple loop~$\gamma$ of~$M$ admits two sides defined as the connected components of $M \setminus \gamma$.
A nontrivial simple geodesic loop~$\gamma$ is said to be minimizing with respect to one of its sides if every pair of points of $\gamma$ can be joined by a minimizing segment which does not pass through this side. A simple geodesic loop minimizing with respect to both sides is said to be minimizing. \\

Now, we can state the following result which immediately leads to Proposition~\ref{prop:L18} when $N=M$ and $\partial N = \emptyset$.

\begin{propositiona}\label{prop:L}
Let $M$ be a Riemannian two-sphere and $N \subset M$ be a connected domain whose boundary $\partial N$ is a finite union of minimizing geodesic loops~$c_i$.
Then
$$ \FillRad(M) \geq \frac{1}{18} L(N).$$
\end{propositiona}

Without loss of generality, we can assume the metric on~$M$ is bumpy (a generic condition).
The curve-deformation process described below plays a key role in the proof of the above proposition.

Let $\gamma$ be a non-minimizing simple geodesic loop on~$M$ and $c$ be a minimizing arc intersecting $\gamma$ only at its endpoints.
The arc $c$ divides $\gamma$ into $c^+$ and $c^-$.
The simple loops $c^+ \cup c$ and $c^- \cup c$, endowed with the orientations induced by~$\gamma$, bound convex domains and converge to simple geodesic loops $\gamma^+$ and $\gamma^-$ through homotopies $z^+_t$ and $z^-_t$ given by the curvature flow.
Applying the curve-shortening process of~\cite[\S2]{Sa04} if necessary, we can extend the homotopies and suppose that $\gamma^+$ and $\gamma^-$ are local minima of the mass functional.
Since the curves $c^+ \cup c$ and $c^- \cup c$ do not intersect transversely and bound convex domains, the simple loops $z^+_t$ and $z^-_t$ are disjoint for $t>0$.
The sum $z_t=z^+_t +z^-_t$ defines the curve-deformation process of $\gamma$ with respect to~$c$.
The homotopy of the one-cycles $z_t$ is $\mu$-nonincreasing and satisfies $\mu(z_t) \leq \length(\gamma)$.
In particular, $\length(\gamma^+)$ and $\length(\gamma^-)$ are less than $\length(\gamma)$. \\

In the following two lemmas, we assume that the only minimizing geodesic loops of~$N$ are the connected components of the boundary $\partial N$.

\begin{lemmaa}\label{lem:two-min}
Let $\gamma$ and $\gamma'$ be two disjoint (nontrivial) simple geodesic loops of~$N$.
Suppose that $\gamma$ (resp. $\gamma'$) is minimizing with respect to the side which does not contain $\gamma'$ (resp. $\gamma$). \\
Then, $\gamma$ or $\gamma'$ is a connected component of $\partial N$.
\end{lemmaa}

\begin{proof}
The geodesics $\gamma$ and $\gamma'$ bound a connected domain $N' \subset N$ with $\gamma,\gamma' \subset \partial N'$.
Suppose that $\gamma$ does not lie in $\partial N$.
Since the only minimizing geodesic loops of~$N$ are the connected components of $\partial N$, the loop~$\gamma$ is not minimizing.
Therefore, there exists a minimizing arc $c$ in $N'$ with endpoints in $\gamma$ which does not lie in $\gamma$.
The curve-deformation process applied to $\gamma$ with respect to $c$ yields two disjoint simple geodesic loops of~$N'$.
The union of these two geodesic loops is noted~$z_1$.
Let us define by induction on $k$ a sequence $z_k$ formed of a disjoint union of simple geodesic loops of~$N'$.
This sequence of one-cycles is $\mu$-nonincreasing. Furthermore, each~$z_k$ bounds with~$\gamma$ a connected domain~$D_k$ of~$N'$ with~$\partial D_k = z_k \cup \gamma$.
If one of the geodesic loops of~$z_k$ is not minimizing with respect to the side opposed to $\gamma$, we apply to it the curve-deformation process with respect to this side.
Otherwise, if one of them is not minimizing with respect to the side containing $\gamma$, we apply to it the curve-deformation process.
In both cases, we obtain a new collection~$\mathcal{C}$ of simple geodesic loops which do not intersect each other.
We define $z_{k+1}$ as the union of the nontrivial geodesic loops of~$\mathcal{C}$ which can be joined to~$\gamma$ by paths cutting no loop of~$\mathcal{C}$.
This induction process stops after a finite number $n$ of steps when $z_n$ is only formed of minimizing geodesic loops of~$N'$.
Since there is no minimizing geodesic loop in the interior of~$N$, we have $z_n \subset \partial N$.
Therefore, the geodesic loop~$\gamma'$, which lies between $z_n$ and $\partial N$, is a connected component of $\partial N$.
\end{proof}

\begin{lemmaa}\label{lem:side}
Let $\gamma$ be a simple geodesic loop of $N$ of length less than~$L(N)$ and suppose that $\gamma$ is minimizing with respect to no side. \\
There exists a non-minimizing simple geodesic loop of~$N$, lying in a side of~$\gamma$, which is minimizing with respect to the side opposed to~$\gamma$.
\end{lemmaa}

\begin{proof}
Assume that every simple geodesic loop of~$N$, disjoint from $\gamma$ and minimizing with respect to the side opposed to~$\gamma$, is minimizing.
Let us recall that the only minimizing geodesic loop of $N$ are the connected components of $\partial N$.
We apply the curve-deformation process to $\gamma$ with respect to each of its sides~$N_1$ and~$N_2$.
Then, we extend these homotopies as in the proof of Lemma~\ref{lem:two-min}.
By assumption, one of these homotopies lies in~$N_1$ and the other in~$N_2$.
Since there is no minimizing geodesic loop in the interior of~$N$, these homotopies, put together, yield a one-parameter family of one-cycles~$z_t$ in~$N$, which starts from $\partial N_1 \cap \partial N$, passes through $\gamma$ and ends at $\partial N_2 \cap \partial N$.
Furthermore, the family~$(z_t)$ induces a generator of $H_2(N,\partial N)$ and satisfies $\mu(z_t) \leq \length(\gamma)$.
Therefore, $\length(\gamma) \geq L(N)$ from the definition of~$L(N)$.
\end{proof}

\begin{remarka} \label{rem:nonmin}
From Lemma~\ref{lem:two-min}, the non-minimizing simple geodesic loops of Lemma~\ref{lem:side} lie in only one side of~$\gamma$.
\end{remarka}

Using the lemma~\ref{lem:side} and the remark~\ref{rem:nonmin}, we can define the side of deformation of a geodesic loop and make the curve-deformation process more precise. \\

\begin{definitiona}
The side of deformation of a non-minimizing simple geo\-desic loop~$\gamma$ of $N$ of length less than~$L(N)$ is
\begin{enumerate}
\item the opposite side with respect to which $\gamma$ is minimizing, if $\gamma$ is minimizing with respect to a side.
\item the side which does not contain any geodesic loop minimizing with respect to the side opposed to~$\gamma$ (except the connected components of~$\partial N$), otherwise.
\end{enumerate}
\end{definitiona}

From now on, when we will apply the curve-deformation process to such a geodesic loop~$\gamma$, it will always be with respect to a minimizing arc with endpoints in $\gamma$ lying in the deformation side of $\gamma$.

\begin{proof}[Proof of Proposition~\ref{prop:L}]
Let us assume first that the only minimizing geodesic loops of~$N$ are the connected components of $\partial N$.
We argue by contradiction as in the proof of Proposition~\ref{prop:ell}.
Pick two reals $r$ and~$\eps$ such that $\FillRad(M) < r < r+\eps <\frac{1}{8} L(M)$.
Let~\mbox{$i:M \hookrightarrow L^{\infty}(M)$} be the Kuratowski distance preserving embedding defined by $i(x)(.)=d_{M}(x,.)$.
By definition of the filling radius, there exists a map $\sigma:P \to U_{r}$ from a $3$-complex~$P$ to the $r$-tubular neighborhood $U_{r}$ of~$i(M)$ in~$L^{\infty}=L^{\infty}(M)$ such that
$\sigma_{|\partial P}: \partial P \to i(M)$ represents the fundamental class of~$i(M) \simeq M$ in $H_{2}(i(M);\Z) \simeq H_{2}(M;\Z)$.
Let us show that the map $\sigma_{|\partial P}: \partial P \to i(M)$ extends to a map \mbox{$f:P \to i(M)$}. \\

Deforming $\sigma$ and subdividing $P$ if necessary, we define an extension $f:P^{1} \cup \partial P \to i(M) \simeq M$ of~$\sigma_{|\partial P}$ to the $1$-skeleton of~$P$ as in the proof of Proposition~\ref{prop:ell}.
From now on, we will identify $i(M)$, $i(N)$ and~$i(\partial N)$ with $M$, $N$ and~$\partial N$.
Recall that the image under~$f$ of the boundary of every $2$-simplex of~$P$ is a simple loop of length less than~$3\rho$, where $\rho=2r + \eps$.
Since the connected components of $\partial N$ are minimizing geodesic loops, the images of the edges of $P$ cut $\partial N$ at most twice.
In particular, the images of the boundary of the $2$-simplices of $P$ cut $\partial N$ at most six times.

On every edge of $P$ whose image cuts $\partial N$, we introduce new vertices given by the preimages of the intersection points with~$\partial N$.
Let~$\Delta^2$ be a $2$-simplex of~$P$.
Every pair of new vertices of $\Delta^2$ which map to the same connected component of~$\partial N$ defines a new edge in $\Delta^2$.
By definition, the images by~$f$ of these new edges are the minimizing segments (lying in $\partial N$) joining the images of their endpoints.
We number the vertices of $P$, old and new, and consider the natural induced order.
The new edges of $\Delta^2$, which map to~$N$, bound a domain $D^2$ of $\Delta^2$.
We introduce new edges on $D^2$ joining the greatest vertex of $D^2$ to every other vertex of $D^2$.
This decomposes~$\Delta^2$ into triangles.
By definition, each new edge maps onto a minimizing segment of~$N$ joining the images of its endpoints.
The triangles of the faces of every $3$-simplex~$\Delta^3$ whose boundaries map to~$N$ bound a domain~$D^3$ of~$\Delta^3$.
We introduce, as previously, new edges on $D^3$ joining the greatest vertex of~$D^3$ to every other vertex of~$D^3$.
These new edges define new faces in~$\Delta^3$.
These new faces  decompose~$\Delta^3$ into $3$-simplices.
As previously, each new edge maps onto a minimizing segment of~$N$ joining the images of its endpoints. \\

In conclusion, we have defined a new simplicial structure on $P$, a map \linebreak
$f:P^1 \cup \partial P \longrightarrow M$ with $f_{|\partial P} = \sigma_{|\partial P}$ and a decomposition of $P$ into two finite subcomplexes $P'$ and $P''$ such that
\begin{enumerate}
\item the image of the 1-skeleton of every 3-simplex of $P'$ (resp. $P''$) lies in the closure of~$N$ (resp. $M \setminus N$).
\item the boundary of every 2-simplex of $P'$ maps into a geodesic triangle whose length of the edges is less than $\frac{3}{2} \rho < \frac{1}{6} L(N)$
\end{enumerate}

We want to extend $f$ to $P^2$.
Let $\partial \Delta^2$ be the boundary of a $2$-simplex $\Delta^2$ of $P$ which does not lie in~$\partial P$.
If $\partial \Delta^2$ lies in~$P''$, we fill its image by the disk it bounds in $M \setminus N$.
Otherwise, the image of $\partial \Delta^2$ lies in $N$.
In this case, it converges to a simple geodesic loop~$c$ of~$N$ by the curve-shortening process of~\cite[\S2]{Sa04}.
We want now to define a $\mu$-nonincreasing homotopy of one-cycles in $N$ from~$c$ to some connected components of $\partial N$, which gives rise to a filling of the image of $\partial \Delta^2$ in $M$. \\

The images of the boundaries of the $2$-simplices of $P'$ converge to simple geodesic loops of~$N$.
The set formed of these simple geodesic loops which are not minimizing is noted~$\mathcal{C}$.
It is finite since $P'$ is a finite simplicial complex.
We want to define homotopies of one-cycles in $N$ from every loop of~$\mathcal{C}$ to some minimizing geodesic loops.
Let us consider the deformation sides of two disjoint geodesic loops of $\mathcal{C}$.
From Lemma~\ref{lem:two-min}, either one is contained in the other or their intersection is disjoint.
Therefore, there exists $\gamma_0 \in \mathcal{C}$ such that no geodesic loop of~$\mathcal{C}$ lies in the deformation side of $\gamma_0$.
The curve-deformation process applied to $\gamma_0$ yields a homotopy of one-cycles $z_t$ defined for $0 \leq t \leq 1$ between $\gamma_0$ and the disjoint union $\gamma^+_0 \cup \gamma^-_0$ of two simple geodesic loops.
The flow of every other loop of $\mathcal{C}$ is constant for $0 \leq t \leq 1$.
We repeat this construction with the set of the non-minimizing geodesic loops of $\{ \gamma \mid \gamma \in \mathcal{C}, \gamma \neq \gamma_0 \} \cup \{\gamma^+_0, \gamma^-_0 \}$.
Since for a (generic) bumpy metric there are only finitely many geodesic loops of length uniformly bounded, this process stops after a finite number of iterations.
From Lemma~\ref{lem:two-min}, the deformation side of every simple geodesic loop of~$N$ lying in the deformation side of a geodesic loop~$\gamma$ does not contain~$\gamma$.
Therefore, the iterations of this process give rise to $\mu$-nonincreasing homotopies of one-cycles between the geodesic loops~$\gamma$ of~$\mathcal{C}$ and the connected components of $\partial N$ which lie in the deformation side of $\gamma$.
A contraction of these latter into points in $M \setminus N$ yields a filling of the boundary of the $2$-simplices of $P'$ in~$M$.
Furthermore, the homotopies arising from two geodesic loops of~$\mathcal{C}$ which do not intersect each other remain disjoint at every time $t$, except possibly for some connected components lying in $\partial N$.
Thus, the sum $z_t$ of the homotopies arising from two disjoint geodesics $\gamma$ and $\gamma'$ of $\mathcal{C}$ satisfies $\mu(z_t) \leq \length(\gamma) + \length(\gamma') \leq 9 \rho$.

In conclusion, for every $2$-simplex $\Delta^2$ of $P$, we get a degree one map from~$\Delta^2$ onto its image in~$i(M) \simeq M$ which agrees with $f$ on the boundary.
This yields the desired extension $f:P^2 \longrightarrow i(M) \simeq M$. \\

Let us extend this map to $P^3$.
Let $\Delta^3$ be a $3$-simplex of $P$.
The restriction of $f$ to the boundary $\partial \Delta^3$ is noted $\varphi : \partial \Delta^3 \longrightarrow M$.
If $\Delta^3$ lies in $P''$, it extends to a map defined on $\Delta^3$ whose image lies in the closure of~$M \setminus N$.
Otherwise, $\Delta^3$ lies in $P'$ and the map $\varphi$ induces a class $[\partial \Delta^3]$ in $H_2(M, M \setminus N) \simeq H_2(N, \partial N)$.
This class arises from degree one maps defined on the faces of~$\Delta^{3}$.
Therefore, it is trivial or generates $H_{2}(N,\partial N)$.
The map~$\varphi$ extends to $\Delta^3 \longrightarrow M$ if and only if the class $[\partial \Delta^3]$ vanishes in $H_2(N, \partial N)$.
Since the images of the edges of $\Delta^3$ are minimizing segments, the image of the $1$-skeleton of~$\Delta^3$ is isotopic to one of the following two graphs on~$M$ shown in Figure \ref{graphs}.

Therefore, the faces of $\Delta^3$ decompose into two pairs such that the edges of the faces of each pair do not intersect each other transversely.
By construction, the images of the faces of $\partial \Delta^3$ are defined by $\mu$-nonincreasing homotopies of one-cycles $z_i^t$ between the images of their boundaries and some connected components of $\partial N$ or null one-cycles.
Furthermore, the sum of the two homotopies $z_1^t$ and $z_2^t$ arising from a pair of faces that do not intersect each other transversely satisfies $\mu(z_1^t +z_2^t) \leq 9 \rho$.
The same goes for the other pair of faces, which yields the homotopies $z_3^t$ and~$z_4^t$.
Putting together the homotopies $z_1^t + z_2^t$ and $z_3^t + z_4^t$, we obtain a one-parameter family of one-cycles $z_t$.
By construction, the family~$(z_t)$ represents $[\partial \Delta^3]$ and satisfies $\mu(z_t) \leq 9 \rho < L(N)$.
Therefore, the class $[\partial \Delta^3]$ vanishes in $H_2(N, \partial N)$. \\

Thus, the map $f$ extends to each simplex of $P$ and gives rise to an extension $f:P \longrightarrow i(M) \simeq M$ of~$\sigma_{|\partial P}$.
This proves the inequality when $N$ has no minimizing geodesic in its interior. \\

Suppose now that $N$ admits a minimizing geodesic loop~$\gamma$ which does not lie in $\partial N$.
This simple loop decomposes $N$ into two connected components $N_1$ and $N_2$, whose boundaries are finite unions of minimizing simple geodesic loops.
By induction on the number of nontrivial minimizing geodesic loops which lie in the interior of $N$, we have $\FillRad(M) \geq \frac{1}{18} L(N_i)$ for $i=1,2$.
Let us consider some homotopies of one-cycles on $N_i$ between the connected components of $\partial N_i$ with $i=1,2$ which induce generators of $H_2(N_i, \partial N_i)$ for $i=1,2$.
Put together, these homotopies yield one-parameter families of one-cycles satisfying (C.1-2).
We deduce that $L(N) \leq \max \{L(N_1), L(N_2)\}$.
This proves the result by induction.
\end{proof}

\begin{figure}[h]
\begin{center}
\AffixLabels{\centerline{\epsfig{file =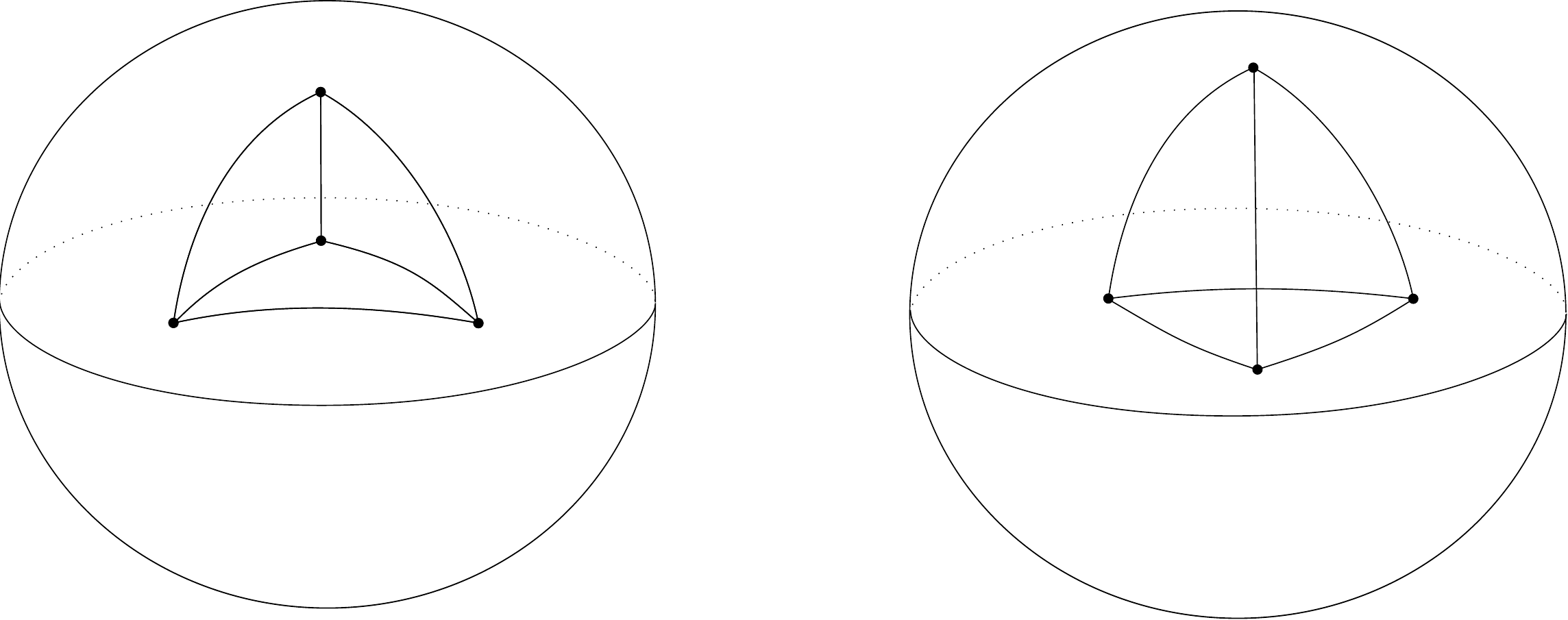,width=10cm,angle=0}}}
\end{center}
\caption{}
\label{graphs}
\end{figure}

\end{document}